\documentclass[11pt]{amsart}
\pdfoutput=1
\usepackage{amssymb}
\usepackage{amsmath}
\usepackage{amsfonts}
\usepackage[usenames]{color}
\usepackage{graphicx}
\usepackage{array}

\makeatletter
 
 \@addtoreset{equation}{section}
\makeatother

\newtheorem{theorem}{Theorem}[section]
\newtheorem{lemma}[theorem]{Lemma}
\newtheorem{assumption}[theorem]{Assumption}
\newtheorem{proposition}[theorem]{Proposition}

\newtheorem{claim}[theorem]{Claim}
\newtheorem{corollary}[theorem]{Corollary}

\theoremstyle{definition}
\newtheorem{definition}[theorem]{Definition}
\newtheorem{example}[theorem]{Example}

\theoremstyle{remark}
\newtheorem{remark}[theorem]{Remark}

\numberwithin{equation}{section}

\newcommand{\Z}{\mathbb{Z}}
\newcommand{\R}{\mathbb{R}}

\newcommand{\s}{\sigma}
\newcommand{\be}{\beta}
\newcommand{\de}{\delta }
\begin{document}

\title{The Self-Linking Number in Annulus and Pants Open Book Decompositions}

\author{Keiko Kawamuro}
\address{Department of Mathematics, The University of Iowa, Iowa city, Iowa 52240}
\email{kawamuro@uiowa.edu}

\author{Elena Pavelescu}
\address{Department of Mathematics, Rice University, Houston, Texas 77005}
\email{Elena.Pavelescu@rice.edu}
\thanks{The first author was partially supported by NSF grants DMS-0806492 and DMS-0635607. }


\date{November 27, 2010}


\begin{abstract}

We find a self-linking number formula for a given null-homologous transverse link in a contact manifold that is compatible with either an annulus or a pair of pants open book decomposition.
It extends Bennequin's self-linking formula for a braid in the standard contact $3$-sphere.

\end{abstract}

\maketitle

\section{Introduction}

Alexander's theorem  \cite{A} states that every closed and oriented $3$-manifold admits an open book decomposition.

\begin{definition}
Let $\Sigma$ be a surface with non empty boundary and $\phi$ be a diffeomorphism of the surface fixing the boundary pointwise.
We construct a closed manifold
\begin{equation*}
M_{(\Sigma, \phi)} = \Sigma \times [0,1] / \sim
\end{equation*}
where ``$\sim$'' is an equivalence relation satisfying $(\phi(x), 0) \sim (x,1)$ for $x\in {\rm Int}(\Sigma)$ and $(x, \tau) \sim (x, 1)$ for $x \in \partial \Sigma$ and $\tau \in [0, 1]$.
The pair $(\Sigma, \phi)$ is called an {\em abstract open book decomposition} of the
manifold $M_{(\Sigma, \phi)}$.

Alternatively, an \emph{open book decomposition} for $M$ can be defined as a pair (L, $\pi$), where
(1) $L$ is an oriented link in $M$ called \textit{the binding} of the open book;
(2) $\pi : M\setminus L\to S^1$ is a fibration whose fiber, $\pi^{-1}(\theta)$, called a \textit{page}, is the interior of a compact surface $\Sigma_{\theta}\subset M$ such that $\partial \Sigma_\theta = L$ for all $\theta\in S^1$.

\end{definition}

One of the central results about the topology of contact $3$-manifolds is Giroux correspondence~\cite{G}:
$$\left\{
\begin{array}{l}
\mbox{contact structures $\xi$ on $M^3$} \\
\mbox{up to contact isotopy}
\end{array}
\right\}
\stackrel{1-1}{\longleftrightarrow}
\left\{
\begin{array}{l}
\mbox{open book decompositions $(\Sigma, \phi)$} \\
\mbox{of $M^3$ up to positive stabilization}
\end{array}
\right\}.
$$
For example, the standard contact structure $\xi_{std}=\ker(dz+r^2 d\theta)$ on $S^3=\R^3 \cup \{\infty\}$ corresponds to the open book decomposition $(D^2, id)$.

We define a braid  and the braid index in a general open book setting:

\begin{definition}\label{def-braid}
Suppose $(L, \pi)$ is an open book decomposition for a $3$-manifold $M$.
A link $K \subset M$ is called a (closed) {\em braid} if $K$ transversely intersects each page $\Sigma_\theta =\pi^{-1}(\theta)$ of the open book.
That is, at each point $p\in K \cap \Sigma_\theta$, we have $T_p\Sigma_\theta\oplus T_pK=T_pM$.
The {\em braid index} of a braid $K$ is the degree of the map $\pi$ restricted to $K$. In other words, if a braid $K$ intersects each page in $n$ points, then the braid index of $K$ is $n$.
\end{definition}

Bennequin~\cite{Ben} proved that any transverse link in $(S^3, \xi_{std})$ can be transversely isotoped to a closed braid in $(D^2, id)$.
Later the second author generalized Bennequin's result into the following:

\begin{theorem}\cite[Theorem 3.2.1]{P}\label{elena}
Suppose $(\Sigma, \phi)$ is an open book decomposition for a $3$-manifold
$M=M_{(\Sigma, \phi)}$. Let $\xi=\xi_{(\Sigma, \phi)}$ be a compatible contact structure.
Let $K$ be a transverse link in $(M, \xi)$. Then $K$ can be transversely isotoped to a braid in $(\Sigma, \phi)$.
\end{theorem}

The {\em self linking $($Bennequin$)$} number is a classical invariant for transverse knots.
Bennequin \cite{Ben} gave a formula of the self linking number for a braid $b$ in $(D^2, id)$:
\begin{equation}\label{Bennequin-eq}
sl(b)=-n+a,
\end{equation}
where $n$ is the braid index, and $a$ the algebraic crossing number (the exponent sum) of the braid.

The first goal of this paper is to give a combinatorial description for the self linking number of a null-homologous transverse link in the contact lens spaces compatible with $(A, D^k)$ the annulus $A$ open book decomposition with monodromy the $k^{th}$ power of the positive Dehn twist $D$.
By Theorem~\ref{elena}, our problem is reduced to searching a self linking formula for a null-homologous braid in the open book decomposition $(A, D^k)$.
Such a braid is given by a product of permutations of points in a local disk on the annulus $A$ and moves of points which turn around the hole of $A$. We denote by $a_{\sigma}$ the algebraic crossing number of the local permutations, and by $a_{\rho}$ the algebraic rotation number around the hole of $A$, see Definition \ref{def of a_rho} for precise definitions. With these notations, we extend Bennequin's formula (\ref{Bennequin-eq}) into the following:

\vspace{3mm}

\noindent\textbf{Theorem~\ref{sl-thm}. }
{\em
Let $b$ be a null-homologous closed braid in $(A, D^k)$ of braid index $n$.
For $k\ne 0$ we have
$$ sl(b)= -n+a_{\sigma}+a_{\rho}(1-\frac{a_\rho}{k}).$$ 
When $k=0$ there exists a canonical Seifert surface $\Sigma_b$ of $b$ and we have
$$sl(b,\Sigma_b)= -n+a_{\sigma}. $$}

The Seifert surface $\Sigma_b$ will be constructed in Section 3. The surface is canonical in the sense that the way of construction is similar to that of the standard Seifert surface, or Bennequin surface, of a closed braid in $S^3$.
\vspace{3mm}

Our second goal is to find a self-linking formula for null-homologous transverse links in a contact Seifert fibered manifold $M$ of signature $(g=0, k_1, k_2, k_3)$.
Let $S$ be a pair of pants (a disk with two holes).
Let $D_i$ ($i=1,2,3$) be the positive Dehn twists along the curves parallel to the boundary circles of $S$.
Then $M$ has an open book decomposition $(S, D_1^{k_1} \circ D_2^{k_2} \circ D_3^{k_3})$, and is equipped with a compatible contact structure.
A braid in the pants open book is a product of permutations of points in a local disk on $S$ and moves of points which turn around the holes of $S$.
We denote by $a_{\sigma}$ the algebraic crossing number of the local permutations and by $a_{\rho_i}$ ($i=2,3$) the algebraic winding number around the holes.
See Definition~\ref{def of a_rho_i} for precise definitions.
We obtain the following formula which also extends (\ref{Bennequin-eq}).

\vspace{3mm}

\noindent\textbf{Theorem~\ref{sl for Seifert manifold}}
{\em
Let $b$ be a null-homologous braid in $(S, D_1^{k_1} \circ D_2^{k_2} \circ D_3^{k_3})$ of braid index $n$.
Suppose $k_1, k_2, k_3$ are integers with
$k_1, k_2, k_3\geq 0$; $k_1, k_2, k_3\leq 0$; or $k_1=0, k_2k_3<0$.
We have:
$$
sl(b, [\Sigma_b])=-n+a_{\sigma}+a_{\rho_2}(1-s_2) + a_{\rho_3}(1-s_3) - (s_2 + s_3) k_1,
$$
where $\Sigma_b$ is some Seifert surface for $b$. The constants $s_2$, $s_3$ are determined by $a_{\rho_2}$,    $a_{\rho_3}$, $k_1$, $k_2$ and $k_3$, under the assumption that $b$ is null-homologous, see Definition~\ref{def of a_rho_i}.
}

\vspace{3mm}

The organization of the paper is the following:

In Section~\ref{preliminaries}, we fix notations and study properties of the contact lens space $(M_{(A, D^k)}, \xi_k)$.

In Section~\ref{section 2}, we construct a Bennequin type Seifert surface $\hat{F}_b$ for a given braid $b$ in $(A, D^k)$.
In general, this $\hat{F}_b$ is an immersed surface and the Bennequin-Eliashberg inequality is not satisfied even for tight cases. 
We resolve all the singularities and obtain an embedded surface $\Sigma_b$.
We develop a theory about resolution of singularities of an immersed surface and corresponding changes in characteristic foliations.

In Section~\ref{section sl}, we prove Theorem~\ref{sl-thm}, an explicit formula of the self linking number relative to  $\Sigma_b$, which extends Bennequin's formula (\ref{Bennequin-eq}).
As the self linking number is defined to be the euler number of the contact $2$-plane bundle relative to the surface framing, we measure the difference between the immersed $\hat{F}_b$-framing and the embedded $\Sigma_b$-framing.
We also study the behavior of our self linking number under a braid stabilization.
Corollary~\ref{thm-stab} states that our self linking number is invariant under a positive stabilization and changes by $2$ under a negative stabilization, which extends Bennequin's result for braids in $(S^3, \xi_{std})$.

In Section~\ref{section-seifert}, we apply our surface construction method to some class of contact Seifert fibered manifolds and prove Theorem~\ref{sl for Seifert manifold}. \\

\textbf{Acknowledgements. }  The authors would like to thank John Etnyre for numerous useful comments and sharing his ideas, especially those on Corollary~\ref{cor-sign}, and Matthew Hedden for helpful comments on Section~\ref{section sl}. They also thank the referee for carefully examining the paper and providing constructive comments.  K.K. thanks Tim Cochran and Walter Neumann for stimulus conversations.

\section{Preliminaries}\label{preliminaries}

Let $A=S^1 \times I$ be an annulus and $D_\alpha$ the positive Dehn twist about the core circle $\alpha=S^1 \times \{\frac{1}{2}\}$.
For simplicity, we denote $D_\alpha$ by $D$.
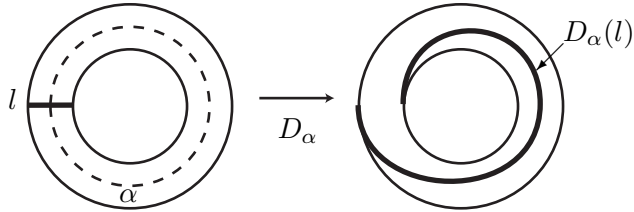
\begin{figure}[htpb!]
\begin{center}
\begin{picture}(203, 78)
\put(0,0){\includegraphics{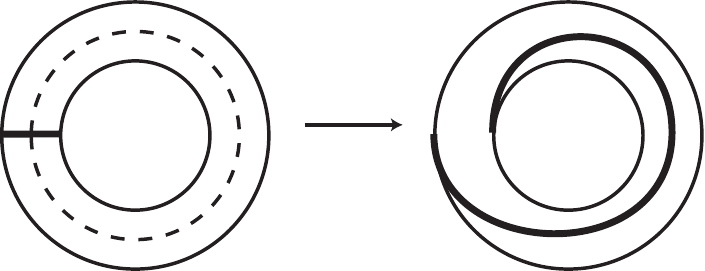}}
\put(35, 3){$\alpha$}
\put(-7, 38){$l$}
\put(95, 27){$D_\alpha$}
\put(203, 63){\vector(-1,-1){10}}
\put(203, 64){$D_\alpha(l)$}
\end{picture}
\caption{A positive Dehn twist $D_\alpha$ about $\alpha$. }\label{rh-twist}
\end{center}
\end{figure}

We study an abstract open book decomposition $(A, D^k)$.

\begin{claim}\label{ob}
The corresponding manifold $M_{(A, D^k)}$ to  $(A, D^k)$ is:
$$
M_{(A, D^k)}= \left\{
\begin{array}{ll}
L(k, k-1) & \mbox{ if } k>0, \\
S^1 \times S^2 & \mbox{ if } k=0, \\
L(|k|, 1) & \mbox{ if } k<0.
\end{array}
\right.
$$
\end{claim}

\begin{proof}
Let $D_\circ \simeq D^2$ be a disk and $\gamma:= \partial D_\circ$.
Recall that $(D_\circ, id)$ is a planar open book decomposition for $(S^3, \xi_{std})$.
Let $D_\mu \subset D_\circ$ be a disc with boundary $\mu$.
The core of the solid torus $D_\mu \times S^1 \subset S^3$ is the unknot, $U$.  The meridian of the torus $T_\mu=\partial(D_\mu \times S^1)$ is $\mu$.
Pick a point $p \in \mu$, and define a longitude $\lambda$ of $T_\mu$ as $\lambda=\{p\} \times S^1$.
Remove $D_{\mu}\times S^1$ from $S^3$, and attach a new solid torus by identifying its meridian $m$ with $\lambda$ and its longitude $l$ with $-\mu$.
This is the $0$-surgery along the unknot $U$.
The resulting manifold is $S^1 \times S^2$.
In this way we get an open book decomposition $(A, id_A)$ for $S^1 \times S^2$, whose page $A$ is the union of the annulus $D_\circ \setminus D_\mu$, shaded in Figure~\ref{surgery}-(1), and the annulus bounded by $-l$ and the core $\gamma'$ of the solid torus, sketched in Figure~\ref{surgery}-(2).
\begin{figure}[htpb!]
\begin{center}
\begin{picture}(396, 108)
\put(0,0){\includegraphics{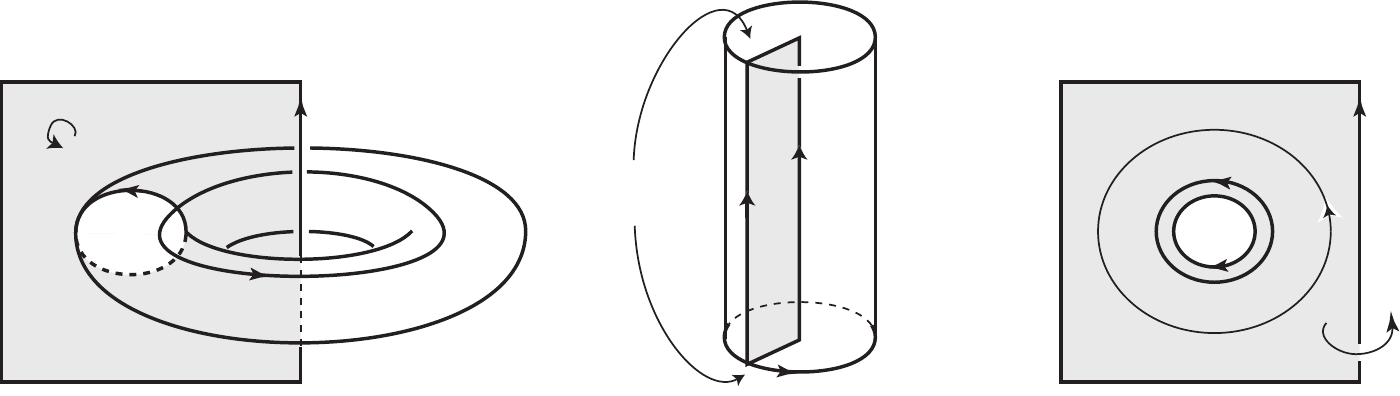}}

\put(-10, 95){(1)}
\put(33, 50){$\mu$} \put(90, 75){$\gamma$}
\put(100, 25){$\lambda$}
\put(135, 45){\small$T_\mu$}

\put(165, 95){(2)}
\put(235, 50){$\gamma'$}
\put(180, 55){$id$}
\put(225, -2){$m$}
\put(218, 70){$l$}

\put(285, 95){(3)}
\put(395, 60){$\gamma$}
\put(335, 65){\small $\mu=-l$} \put(345, 40){$\gamma'$}
\put(400, 27){\small $\lambda$}
\put(340, 15){$U'$}
\end{picture}
\caption{(1) Removing a solid torus $D_{\mu}\times S^1$ from $S^3$. (2) The attaching solid torus. (3) The page annulus $A$.}\label{surgery}
\end{center}
\end{figure}
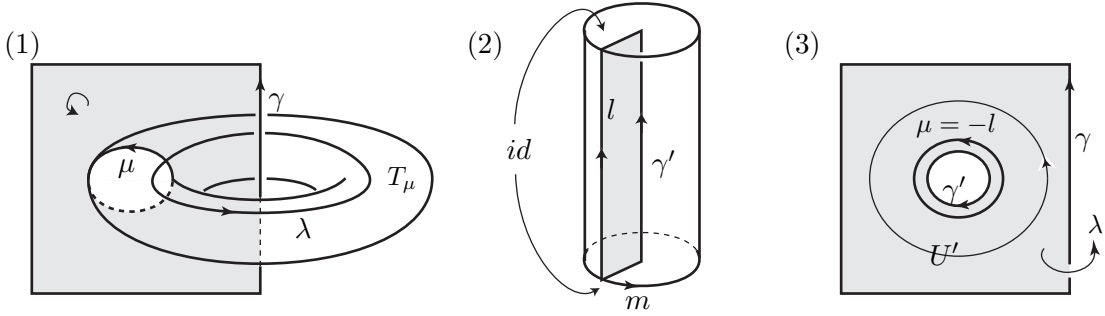

The Dehn twist $D^k$ about the core $U' \subset (D_\circ \setminus D_\mu) \subset A$, sketched in Figure~\ref{surgery}-(3), of the page annulus $A$ is equivalent to applying $(\frac{1}{-k})$-surgery along the unknot $U'$.
The link $(U \cup U') \subset S^3$ is the positive Hopf link.
By the slam-dunk operation, the surgery description is reduced to the $k$-surgery along $U$, which represents $L(k, -1)=L(k, k-1)$ when $k>0$ and $L(|k|, 1)$ when $k<0$.
\end{proof}

Let $(M_{(A, D^k)}, \xi_k)$ be the contact manifold corresponds to the open book $(A, D^k)$.

\begin{claim}\label{goodman}
The contact manifold $(M_{(A, D^k)}, \xi_k)$ is overtwisted if and only if $k <0$.
When $k\geq 0$, this $\xi_k$ is the unique tight contact structure for $L(k, k-1)$.
\end{claim}

\begin{proof}
If $k<0$, Goodman's criterion for overtwistedness~\cite[Theorem 1.2]{Go} implies that $\xi_k$ is overtwisted.

When $k=0$, according to \cite[proof of Lemma 3.2]{EH} of Etnyre-Honda, the open book is a boundary of a positive Lefschetz fibration on a $4$-manifold $X$, so that $(S^1\times S^2, \xi_0)$ is Stein filled by $X$,  hence tight. 
Moreover, $\xi_0$ is the unique tight contact structure on $S^2\times S^1$ due to Eliashberg~\cite {El}.

When $k>0$, the monodromy is a product of positive Dehn twists. 
Etnyre-Honda's \cite[Lemma 3.2]{EH} guarantees that the contact structure compatible with such an open book is Stein fillable, hence tight. The uniqueness for $k>0$ follows from Honda's classification of tight contact structures for lens spaces \cite{H}. More precisely, we have
$$-\frac{k}{k-1} = -2 - \frac{1}{-2 - \frac{1}{-2 - \cdots -\frac{1}{-2}}} = [-2, -2, \cdots, -2], \mbox{ \small repeating $(k-1)$-times}$$
and $| (-2+1) (-2+1) \cdots (-2+1) |=1$, thus the manifold has the unique tight contact structure.
\end{proof}

We fix notations. See Figure~\ref{generators}.
Suppose we have a null-homologous closed braid $b$ of braid index $n$ in the open book $(A, D^k)$.
Let $\gamma \cup \gamma' =\partial A$ whose orientations are induced by that of $A$.
Let $A_\theta$ ($\theta \in [0,1]$) denote the page $A \times \{\theta\} \subset M_{(A, D^k)}$.
Under the identification $A=S^1 \times [0,1]$, we set $\alpha= S^1 \times \{\frac{1}{2}\}$. Let $\beta$ be a circle between $\alpha$ and $\gamma$ which is oriented clockwise.

\begin{figure}[htpb!]
\begin{center}
\begin{picture}(150, 160)
\put(0,0){\includegraphics{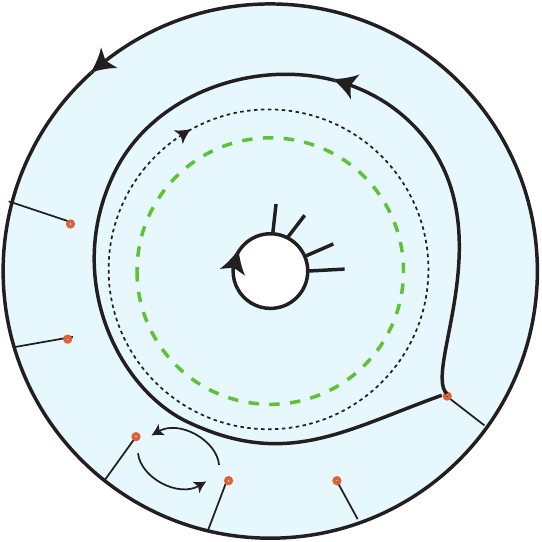}}
\put(150, 117){$\gamma$}
\put(70, 59){$\gamma'$}
\put(110, 70){$\alpha$}
\put(120, 80){$\beta$}
\put(115, 125){$\rho$}
\put(10, 82){$x_1$}
\put(13, 48){$x_2$}
\put(29, 29){$x_i$}
\put(40, 12){$\sigma_i$}
\put(65, 10){$x_{i+1}$}
\put(132, 40){$x_n$}
\put(-6, 91){$l_1$}
\put(-3, 51){$l_2$}
\put(25, 11){$l_i$}
\put(58, -10){$l_{i+1}$}
\put(140, 25){$l_n$}
\put(92, 70){$u_s$}
\put(68, 100){$u_1$}
\put(85, 97){$u_2$}
\end{picture}
\caption{}\label{generators}
\end{center}
\end{figure}

\begin{assumption}\label{position-assumption}

Choose points  $x_1, \cdots, x_n$ sitting {\em between} $\gamma$ and $\alpha$. By braid isotopy, which preserves the transverse knot class (Theorem~\ref{elena2}-(2)), we may assume that:
\begin{equation*}
b \cap A_0 = \{x_1, \cdots, x_n\}.
\end{equation*}

\end{assumption}

Let $\sigma_i$ ($i=1, \cdots, n-1$) be the generators of Artin's braid group $B_n$ satisfying $\sigma_i \ \sigma_{i+1}\ \sigma_i = \sigma_{i+1}\ \sigma_i\ \sigma_{i+1}$ and $\sigma_i \sigma_j = \sigma_j \sigma_i$ for $|i-j|\geq 2$.
Geometrically, $\sigma_i$ acts by switching the marked points $x_i$ and $x_{i+1}$ counterclockwise.
The circle $\beta$ will appear in Section~\ref{section F_b}.
Let $\rho$ be a braid element which moves $x_n$ once around the annulus in the indicated direction.

\begin{proposition}
An $n$-strand braid $b$ in $(A, D^k)$ has a braid word in $\{\s_1, \cdots, \s_{n-1}, \rho \}$.
\end{proposition}

\begin{proof}
Let $D_\bullet \subset D_\circ$ be concentric disks of center $o$.
Identify the annulus $A$ with $D_\circ \setminus D_\bullet$ and $\partial D_\bullet = -\gamma',$ $\partial D_\circ = \gamma$.
Consider the union $\tilde{b}:=b \cup (\{o\} \times [0,1]) \subset D_\circ \times [0,1]$, which we identify with an $(n+1)$-strand braid in Artin's braid group $B_{n+1}$.
Let $p: D_\circ \times [0,1] \to D_\circ$ be the projection onto the first factor.
Up to homotopy, we can think that $p(\{o\} \times [0,1])$ is a (non-simple) closed  curve in $D_\circ \setminus \{ x_1, \cdots, x_n \}$.
Denote its homotopy class by
$$b_\bullet := [p(\{o\} \times [0,1])] \in \pi_1(D_\circ\setminus \{ x_1, \cdots, x_n \}, o).$$
Let $\rho_1, \cdots, \rho_n$ be generators of $\pi_1(D_\circ\setminus \{ x_1, \cdots, x_n \}, o)$ as in Figure~\ref{generators-2}-(1).
\begin{figure}[htpb!]
\begin{center}
\begin{picture}(340, 125)
\put(0,0){\includegraphics{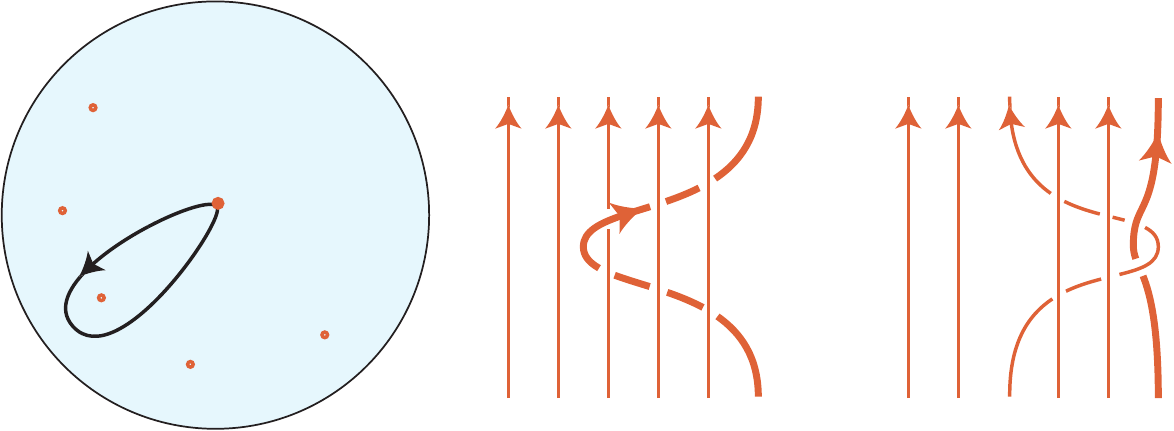}}
\put(25, 100){$x_1$}
\put(32, 37){$x_i$}
\put(90, 32){$x_n$}
\put(10, 45){$\rho_i$}
\put(66, 69){$o$}
\put(144, 100){$1$}
\put(260, 100){$1$}
\put(175, 100){$i$} \put(290, 100){$i$}
\put(203, 100){$n$} \put(318, 100){$n$}
\put(218, 100){$o$} \put(333, 100){$o$}
\put(0, 115){(1)} \put(140, 115){(2)} \put(252, 115){(3)}
\end{picture}
\caption{}\label{generators-2}
\end{center}
\end{figure}
The transition from Figure~\ref{generators-2}-(2) to (3) shows:
\begin{eqnarray*}
\rho_i &=& \sigma_i^{-1} \cdots \sigma_{n-1}^{-1}\ \sigma_n^2\ \sigma_{n-1} \cdots \sigma_i, \quad (i=1, \cdots, n-1), \\ 
\rho_n &=& \sigma_n^2.
\end{eqnarray*}
Since our $\rho=\rho_n$ is equal to $\sigma_n^2$ in the braid group $B_{n+1}$, the  braid $\tilde{b}$ can be written in letters $\{\sigma_1, \cdots, \sigma_{n-1}, \rho\}$.
Since $b \subset \tilde{b}$, the statement of the proposition follows.
\end{proof}

\begin{definition}\label{def of a_rho}
Let $a_\s \in \Z$ (resp. $a_\rho\in \Z$) be the exponent sum of $\s_1, \cdots, \s_{n-1}$'s (resp. $\rho$) in the braid word of $b$.
\end{definition}

\begin{proposition}\label{assumption for a_rho}
If $k> 0$ (resp. $k < 0$), we may assume that $a_\rho \geq 0$ (resp. $a_\rho \leq 0$).
\end{proposition}

To prove Proposition~\ref{assumption for a_rho}, we first define braid stabilization and recall its properties.

\begin{definition}
Let $b$ be a closed braid in an open book $(\Sigma, \phi)$.
Suppose that $\lambda \subset \partial \Sigma$ is one of the bindings of the open book and $p\in (\Sigma_\theta \cap b)$ is a point, see Figure~\ref{stabilization}.
Join $p$ and a point on $\lambda$ by an arc $a \subset (\Sigma_\theta \setminus b)$.
A {\em positive $($negative$)$ stabilization} of $b$ about $\lambda$ along $a$ is pulling a small neighborhood of $p$ of the braid, then adding a positive (negative) kink about $\lambda$ in a neighborhood of $a$.
\end{definition}
\begin{figure}[htpb!]
\begin{center}
\begin{picture}(300, 85)
\put(0,0){\includegraphics{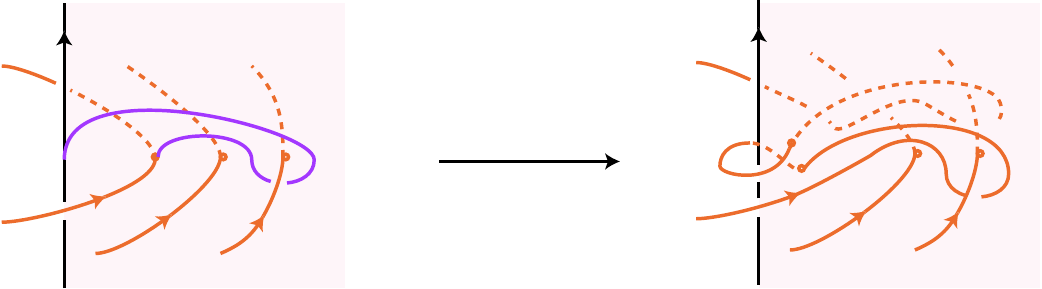}}
\put(118, 49){\small positive}
\put(118, 40){\small stabilization}
\put(60, 50){$a$}
\put(37, 35){$p$}
\put(80, 70){$\Sigma_\theta$}
\put(8, 70){$\lambda$}
\end{picture}
\caption{Positive braid stabilization along $a$.}\label{stabilization}
\end{center}
\end{figure}
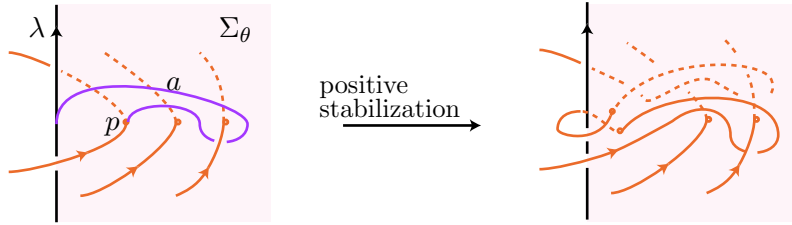

The second author proved Markov theorem in a general open book setting:

\begin{theorem}\cite[Theorem 4.1.3 and 4.1.4]{P}\label{elena2}
\begin{enumerate}
\item Two closed braids $K_1$ and $K_2$ in an open book decomposition have the same topological type if and only if they are related by braid isotopy, positive and negative braid stabilizations.
\item The above $K_1, K_2$ are transversely isotopic if and only if they are related by braid isotopy and positive braid stabilizations.
\end{enumerate}
\end{theorem}

\begin{proof}[Proof of Proposition~\ref{assumption for a_rho}]

Suppose $b$ is an $n$-strand braid.
Recall that $(A, D^k)$ has two binding components, $\gamma$ and $\gamma'$.
Let $a$ be an arc joining $x_n$ and $\gamma'$ and intersecting $\alpha$ at a point as sketched in Figure~\ref{x_n+1}.
\begin{figure}[htpb!]
\begin{center}
\begin{picture}(160, 160)
\put(0,0){\includegraphics{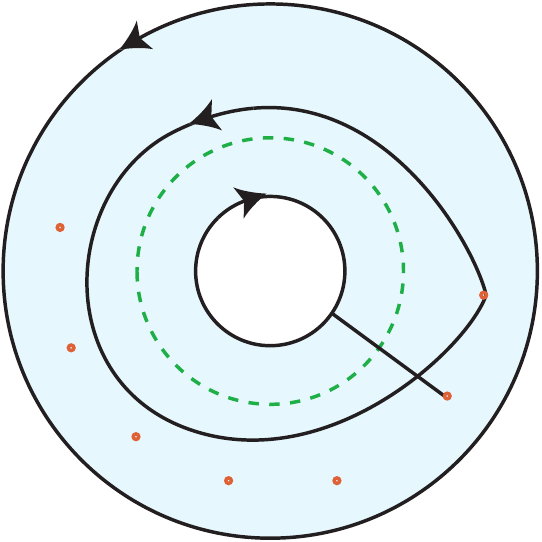}}
\put(16, 140){$\gamma$}
\put(70, 87){$\gamma'$}
\put(113, 80){$\alpha$}
\put(124, 110){$\rho_{n+1}$}
\put(10, 82){$x_1$}
\put(13, 48){$x_2$}
\put(132, 40){$x_n$}
\put(143, 70){$x_{n+1}$}
\put(95, 55){$a$}
\end{picture}
\caption{Definitions of $a, x_{n+1}$ and $\rho_{n+1}$.}\label{x_n+1}
\end{center}
\end{figure}
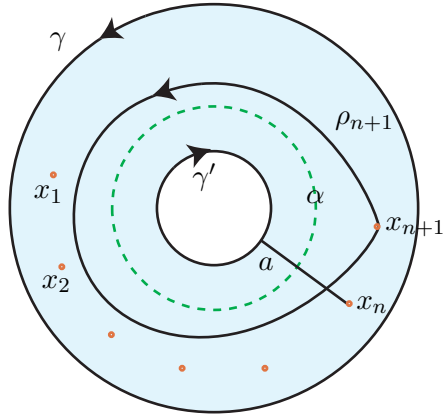
Pick a small line segment of the $n^{\rm th}$ strand in $A \times (1-\epsilon, 1)$, near the top page $A_{\theta=1}$ of the open book, and positively stabilize it along $a$.
As a consequence, it gains a new braid strand, which we call $\nu$, lying in a small tubular neighborhood of $\gamma'$, see Figure~\ref{rho-n}-(1).
\begin{figure}[htpb!]
\begin{center}
\begin{picture}(416, 180)
\put(0,0){\includegraphics{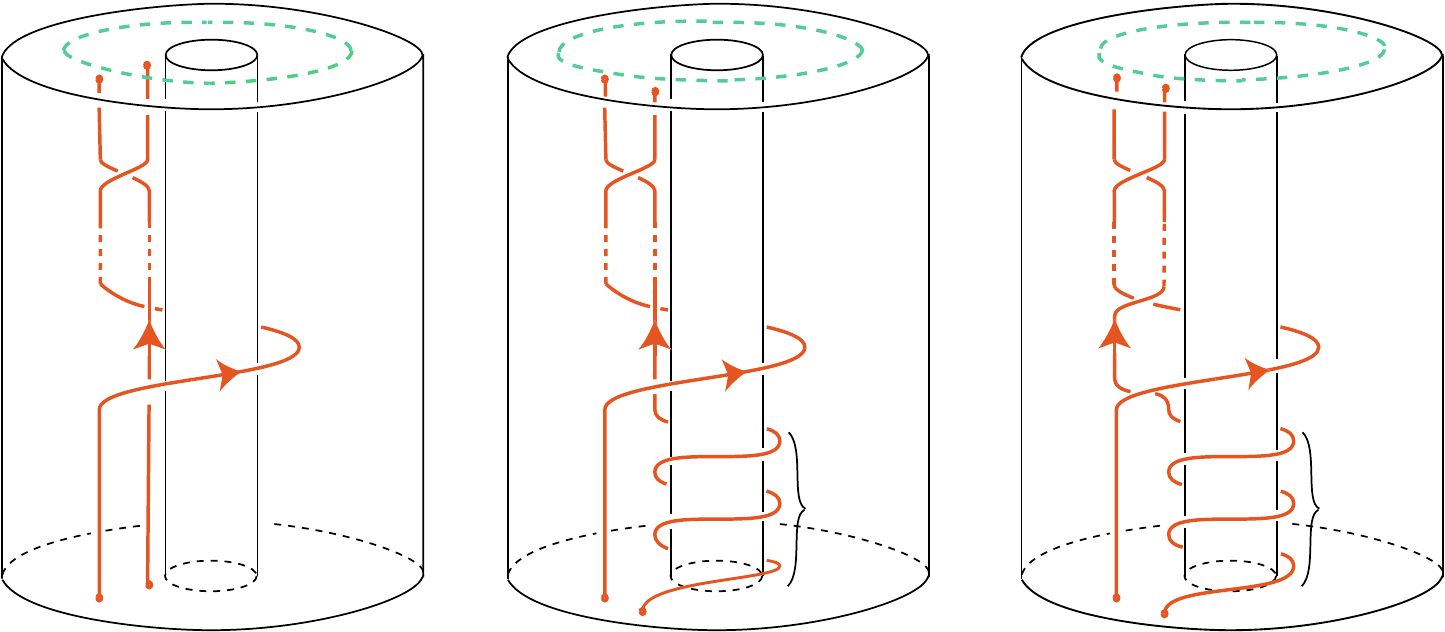}}
\put(0, 179){(1)}
\put(150, 179){(2)}
\put(295, 179){(3)}
\put(25, 165){$x_n$}
\put(166, 165){$x_n$}
\put(310, 165){$x_n$}
\put(40, 167){$\nu$}
\put(55, 160){$\gamma'$}
\put(90, 161){$\alpha$}
\put(235, 161){$\alpha$}
\put(388, 162){$\alpha$}
\put(180, 160){$x_{n+1}$}
\put(325, 160){$x_{n+1}$}
\put(210, 150){$\gamma$}
\put(234, 33){$(\rho_{n+1})^k$}
\put(382, 33){$(\rho_{n+1})^k$}
\put(310, 66){$\sigma_n$}
\put(310, 94){$\sigma_n$}
\put(310, 130){$\sigma_n$}
\put(384, 78){$\rho_{n+1}$}
\put(88, 78){$\rho_n$}
\put(234, 78){$\rho_n$}
\end{picture}
\caption{(1) Positive stabilization about the binding $\gamma'$. (2) Transversely isotope $\nu$ near $\gamma'$ to $x_{n+1}$ near $\gamma$. This introduces $k$ additional $\rho_{n+1}$'s. (3) $\rho_n$ and $\rho_{n+1}$ are related by $\rho_n = \sigma_n \rho_{n+1} \sigma_n$. }\label{rho-n}
\end{center}
\end{figure}
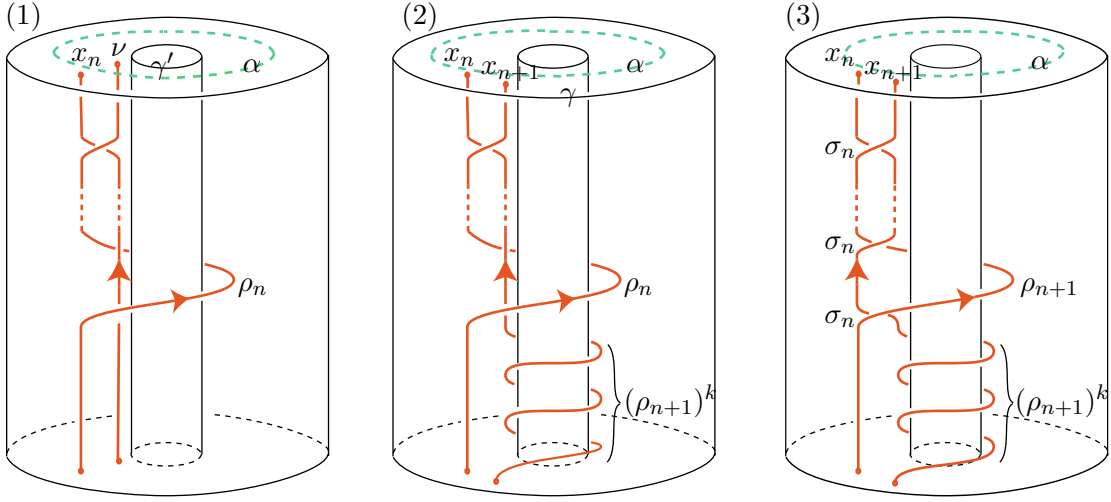
Put a point $x_{n+1} \subset A$ on the right side of $x_n$ between $\gamma$ and define $\rho_{n+1}$ a braid generator as in Figure~\ref{x_n+1}.
Move $\nu$ by a braid isotopy supported in $A \times (1-\epsilon, 1+\epsilon)$ so that $\nu$ intersects the page $A_0=A_1$ at $x_{n+1}$.
This isotopy introduces $(\rho_{n+1})^k$ in $A \times (0, \epsilon)$ as a consequence of the monodromy $D^k$. Compare Figure~\ref{rho-n}-(1) and (2).

We observe that in a stabilized braid, $\rho_{n+1}$ plays the role of the old $\rho=\rho_n$ and as Figure~\ref{rho-n}-(3) shows, they are related by: 
\begin{equation}
\rho_n = \sigma_n \rho_{n+1} \sigma_n.
\end{equation}
Thus a positive stabilization about $\gamma'$ takes a word $b$ to $(\rho_{n+1})^k \ \tilde{b} \ \sigma_{n},$ where $\tilde{b}$ is obtained from $b$ replacing each $\rho$ with $\sigma_n \rho_{n+1} \sigma_n$. The data change in the following way:
$$
n  \to n+1, \ \
a_{\sigma}  \to   a_{\sigma}+1 + 2 a_\rho, \ \
a_{\rho}  \to  a_{\rho}+k.
$$
Theorem~\ref{elena2}-(2) tells that a positive stabilization preserves the transverse knot type, so if $k>0$ (resp. $k<0$) we may assume that $a_\rho\geq 0$ (resp. $a_\rho \leq 0$).
\end{proof}

The next corollary introduces a number $s$:

\begin{corollary}\label{def of s}
If $k\neq 0$ there exists a non-negative integer $s$ such that $a_\rho = sk$. If $k=0$ then $a_\rho =0$.
\end{corollary}

\begin{proof}
In the homology group $H_1(M_{(A, D^k)}, \Z)$, we have $[b] + a_\rho [\beta]=0$. 
Since the braid $b$ is null-homologous $a_\rho [\beta] = [b]=0$.
The meridian $\mu$ introduced in the proof of Claim~\ref{ob} is a generator of $H_1(M_{(A, D^k)}, \Z)=\Z/k\Z$.
Since $a_{\rho}[-\mu]=a_{\rho}[\beta]=0$ we have $a_{\rho}\equiv 0$ (mod $k$), implying the existence of $s\in\Z$ with $a_\rho = sk$ for $k\neq 0$.
Proposition~\ref{assumption for a_rho} guarantees that we may assume $s \geq 0$.
When $k=0$, we have $a_\rho=0.$
\end{proof}


\section{Construction of Seifert surface $\Sigma_b$}\label{section 2}

The goal of this section is to construct a Seifert surface $\Sigma_b$ for a null-homologous braid $b$ whose braid word is written in $\{\s_1, \cdots, \s_{n-1}, \rho \}$.
(By abuse of notation, we use $b$ for both the closed braid and its braid word.)
We first construct a surface $F_b$ and change it to $\tilde{F}_b$. We further deform $\tilde{F}_b \to \check F_b \to \hat{F}_b$ and finally obtain $\Sigma_b$.


\subsection{Construction of the surface $F_b$ }\label{section F_b}

Let $l_i \subset A$ be a line segment perpendicular to $\gamma$ having $x_i$ as one of its endpoints and with the other end on $\gamma$, see Figure~\ref{generators}. Since $l_i$ is disjoint from the Dehn twist curve $\alpha$, in the
resulting manifold, $M_{(A, D^k)}$, the arc $l_i$ swipes a disk
$\de_i:=(l_i \times [0,1])/\sim$. See Figure~\ref{disks}. The center
of $\de_i$ is $l_i \cap \gamma$. We orient $\de_i$ so that the
binding $\gamma$ is positively transverse to $\de_i$.

\begin{figure}[htpb!]
\begin{center}
\begin{picture}(116, 125)
\put(0,0){\includegraphics{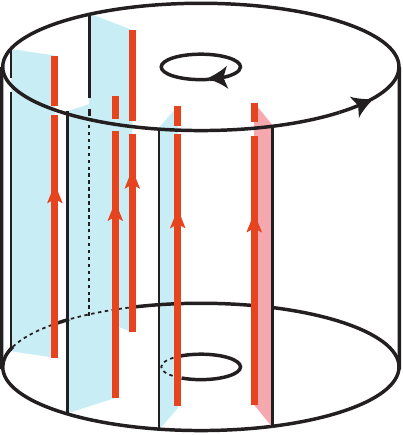}}
\put(29, 105){$\de_1$}
\put(6, 50){$\de_2$}
\put(65, 50){$\de_n$}
\end{picture}
\caption{Oriented disks $\de_1, \cdots, \de_n$. Positive (negative) side is light blue (dark pink). }\label{disks}
\end{center}
\end{figure}
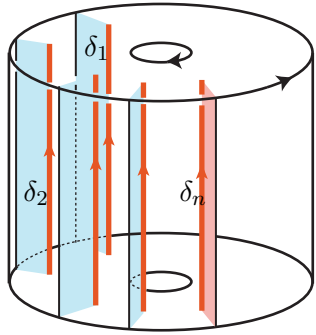

Suppose the braid word for $b$ has length $m$. If the $j^{th}$ ($1
\leq j \leq m$) letter is $\sigma_i$ (resp. $\s_i^{-1}$) then we
join the disks $\de_i$ and $\de_{i+1}$ by a positively (resp.
negatively) twisted band embedded in the set of pages $\{ A_\theta \ | \ \frac{j-1}{m} < \theta < \frac{j}{m}\}$. See Figure~\ref{surface1}-(1).

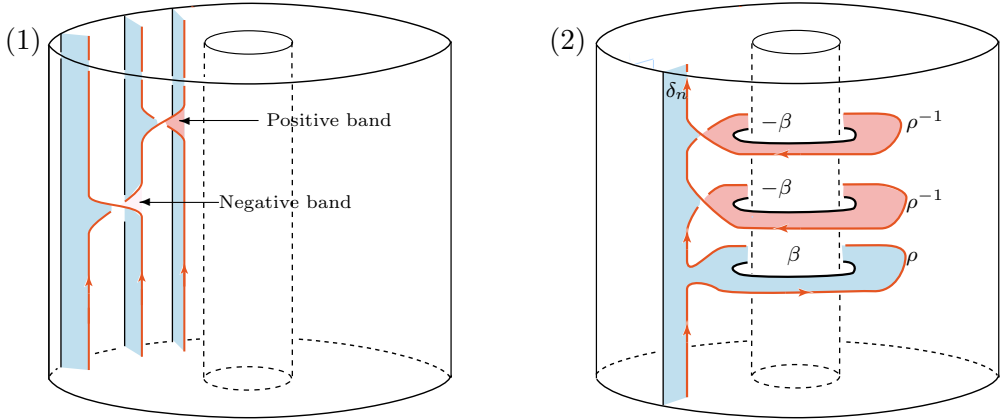
\begin{figure}[htpb!]
\begin{center}
\begin{picture}(360, 157)
\put(0,0){\includegraphics{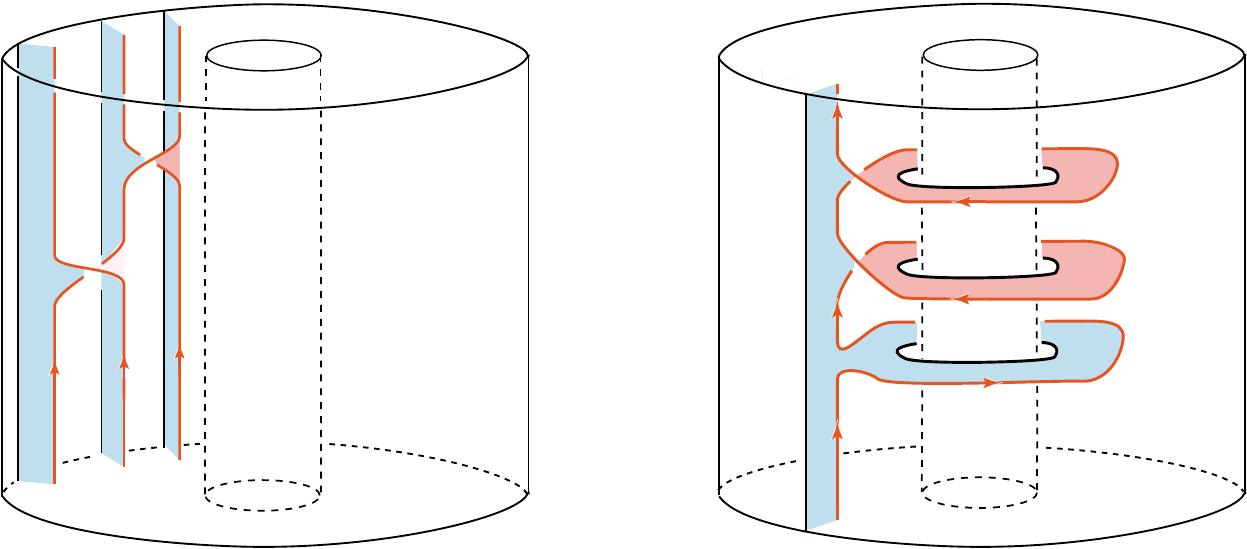}}
\put(-16, 140){(1)}
\put(190, 140){(2)}
\put(325, 110){\scriptsize{$\rho^{-1}$}}
\put(325, 80){\scriptsize{$\rho^{-1}$}}
\put(325, 60){\scriptsize{$\rho$}}
\put(234, 122){\scriptsize{$\de_n$}}
\put(270, 110){\scriptsize{$-\be$}}
\put(270, 85){\scriptsize{$-\be$}}
\put(280, 60){\scriptsize{$\be$}}
\put(80, 113){\vector(-1, 0){30}}
\put(83, 111){\tiny Positive band}
\put(65, 82){\vector(-1, 0){30}}
\put(65, 80){\tiny Negative band}
\end{picture}
\caption{Construction of $F_b$. (1) Twisted bands. (2) ${\mathfrak A}$-annuli. }\label{surface1}
\end{center}
\end{figure}

If the $j^{th}$ letter is $\rho$ (resp. $\rho^{-1}$), then we attach to the disk $\de_n$ an annulus embedded in $\{ A_\theta \ | \ \frac{j-1}{m} < \theta < \frac{j}{m}\}$. We call such an annulus an {\em ${\mathfrak A}$-annulus}.
See Figure~\ref{surface1}-(2). Let $\beta \subset A$ be an oriented circle between circles $\alpha$ and $\gamma$ as sketched in Figure~\ref{generators}.
One of the boundaries of each ${\mathfrak A}$-annulus represents $\rho$ (resp. $\rho^{-1}$) and becomes part of the braid $b$. The other boundary, which we denote by $\beta_j$ (resp. $-\beta_j$), is in $\beta \times (\frac{j-1}{m},\ \frac{j}{m})$.

We call the resulting surface $F_b$.

By \cite[Proposition 4.6.11]{Ge}, we may assume that the
characteristic foliation of our surface is of Morse-Smale type.
Each disk $\de_i$ has a positive elliptic point.
A positive (negative) band between the $\de$-disks contributes one positive (negative) hyperbolic point.
The foliation on the disk $\de_n$ together with an attached ${\mathfrak A}$-annulus has a positive (resp. negative) hyperbolic singularity as sketched in Figure~\ref{foliation}-(1) (resp. (2)) if the corresponding braid word is $\rho$ (resp. $\rho^{-1}).$
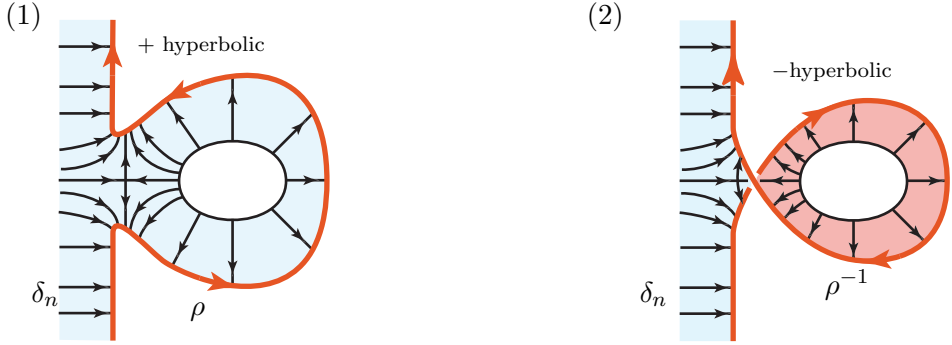
\begin{figure}[htpb!]
\begin{center}
\begin{picture}(360, 130)
\put(0,0){\includegraphics{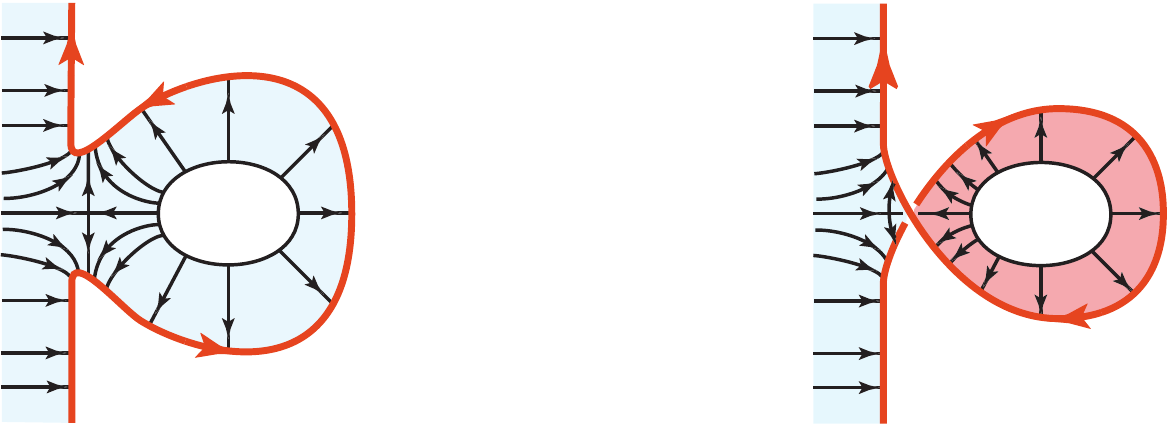}}
\put(-20, 120){(1)} \put(30,110){\scriptsize$+$ hyperbolic}
\put(200, 120){(2)} \put(270, 100){\scriptsize$-$hyperbolic}
\put(-10, 15){$\de_n$} \put(220, 15){$\de_n$}
\put(50, 10){$\rho$}
\put(290, 18){$\rho^{-1}$}
\end{picture}
\caption{Characteristic foliations of ${\mathfrak A}$-annulus for (1) $\rho$ and (2) $\rho^{-1}$.}\label{foliation}
\end{center}
\end{figure}


\subsection{Construction of the surface ${\tilde F}_b$ }\label{section-tilde F_b}

In Section~\ref{section F_b}, we have constructed an embedded oriented surface $F_b$ whose boundary consists of the braid $b$ and copies of $\pm \beta$'s.
Let $a_\s \in \Z$ (resp. $a_\rho\in \Z$) be the exponent sum of $\s_1, \cdots, \s_{n-1}$'s (resp. $\rho$'s) in the braid word for $b$.
Let $r$ be the number of $\rho^{\pm}$'s appearing in the braid word for $b$ of length $m$ (i.e., $0\leq r \leq m$).
Then there exist $1\leq j_1 < \cdots < j_r \leq m$ and $\epsilon_i = \pm 1$ with $\epsilon_1 + \epsilon_2 + \cdots + \epsilon_r = a_\rho$ such that
$$\partial F_b =  b \ \cup \ \epsilon_1\beta_{j_1}  \ \cup \
\epsilon_2\beta_{j_2}\ \cup \cdots \cup \ \epsilon_r\beta_{j_r}.$$

\begin{proposition}\label{a}
By attaching vertical annuli to pairs of $\be$ and $-\be$ circles as described in Figure~\ref{canceling_rho}, we can construct an embedded oriented surface, which we call $\tilde F_b$, whose boundary consists of 
\begin{equation}\label{equation boundary}
\partial \tilde{F}_b = \left\{
\begin{array}{ll}
\mbox{$b$ and $a_{\rho}$ copies of $\beta$}, & \mbox{if $k> 0$, }\\
b, & \mbox{if $k=0$, } \\
\mbox{$b$ and $-a_\rho$ copies of $-\beta$}, & \mbox{if $k<0$.}
\end{array}
\right.
\end{equation}
\end{proposition}

\begin{figure}[htpb!]
\begin{center}
\begin{picture}(350, 74)
\put(0,0){\includegraphics{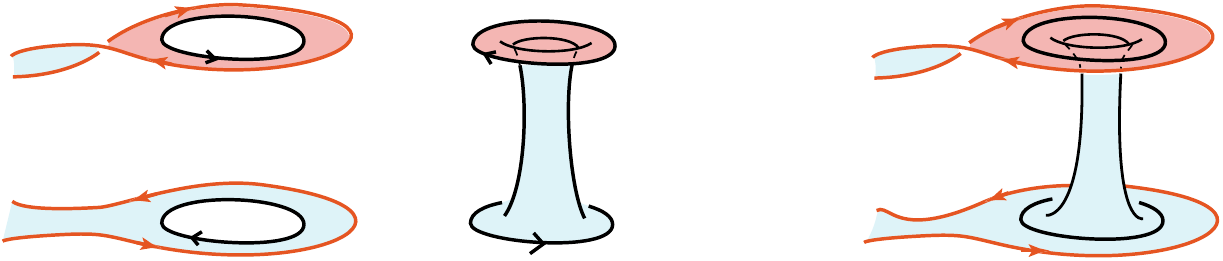}}
\put(-10, 70){\small (1)}
\put(240, 70){\small (2)}
\put(115, 33){$\cup$}
\put(200, 33){$=$}
\put(80, 34){$\beta_{j_i}$}
\put(30, 40){$-\beta_{j_{i+1}}$}
\put(45, 50){\vector(1,1){10}}
\put(80, 30){\vector(-1, -1){10}}
\end{picture}
\caption{Attaching a vertical annulus to ${\mathfrak A}$-annuli.}\label{canceling_rho}
\end{center}
\end{figure}
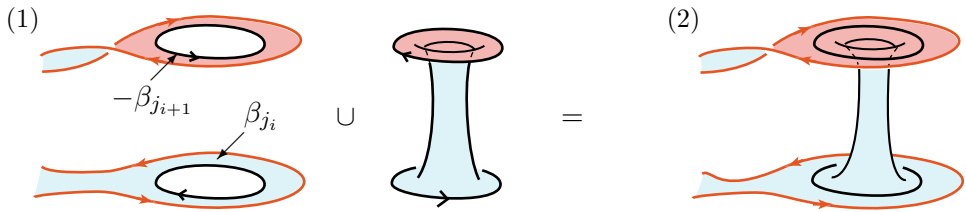

\begin{proof}
Suppose that $\partial F_b =  b \ \cup \
\epsilon_1\beta_{j_1}  \ \cup \ \epsilon_2\beta_{j_2}\
\cup \cdots \cup \ \epsilon_r\beta_{j_r}$ and $\epsilon_1 + \epsilon_2 + \cdots + \epsilon_r = a_\rho.$

If $\epsilon_1 = \epsilon_2 = ... = \epsilon_r$ ($\rho$ and $\rho^{-1}$ do not coexist in the braid word for $b$), then take $\tilde{F_b}= F_b$.

Else, let $1\le i \le r-1$ be the smallest index for which $\epsilon_i \ne \epsilon_{i+1}$. We attach a ``vertical" annulus to $(\epsilon_i\beta_{\j_i} ) \cup (\epsilon_{i+1}\beta_{j_{i+1}})$ as sketched in Figure~\ref{canceling_rho}. 
The boundary of the newly obtained surface (call this surface $F_{b,1}$) contains two less $\pm\beta$-curves:
$$\partial F_{b,1} =  b \ \cup \ 
\epsilon_1\beta_{j_1} \ \cup \cdots \cup 
\epsilon_{i-1}\beta_{j_{i-1}}\cup
\epsilon_{i+2}\beta_{j_{i+2}}\cup \dots  \cup
\epsilon_r\beta_{j_r}$$ 
but it preserves the sum:
$\epsilon_1 + \cdots + \epsilon_{i-1}+ \epsilon_{i+2} + \cdots +\epsilon_r = a_\rho.$
 
Renumber the the boundary components,
$$ \partial F_{b,1} =  b \ \cup \
\epsilon_1\beta_{j_1}  \ \cup \ \epsilon_2\beta_{j_2}\
\cup \cdots \cup \ \epsilon_{r-2}\beta_{j_{r-2}}$$
then repeat the procedure for $F_{b,1}$. 
If $i-1$ is the smallest index for which $\epsilon_{i-1} \neq \epsilon_{i}$, then attach the annulus by nesting it inside the one previously attached. See the right sketch in Figure~\ref{nesting}.
\begin{figure}[htpb!]
\begin{center}
\begin{picture}(330, 125)
\put(0,0){\includegraphics{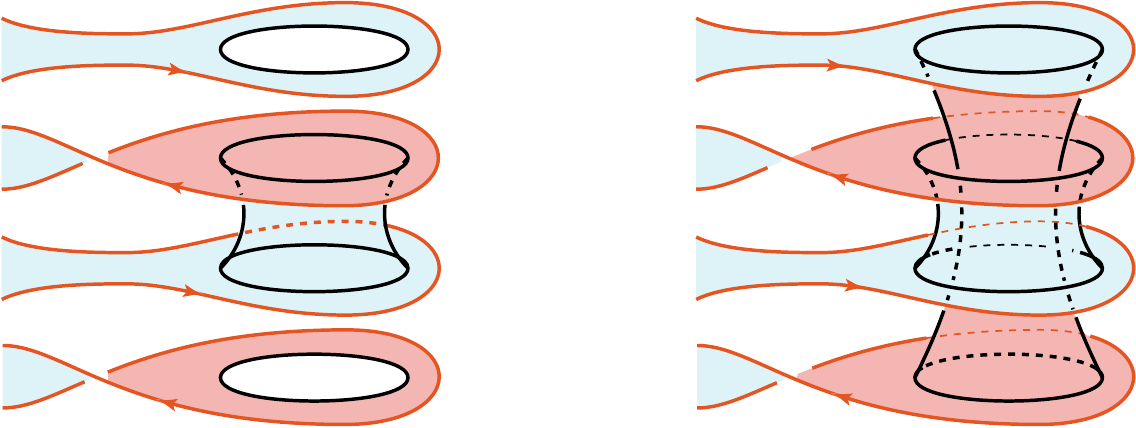}}
\put(130, 105){$i+2$}
\put(130, 70){$i+1$}
\put(130, 40){$i$}
\put(130, 10){$i-1$}
\end{picture}
\caption{ Nested vertical annuli}\label{nesting}
\end{center}
\end{figure}
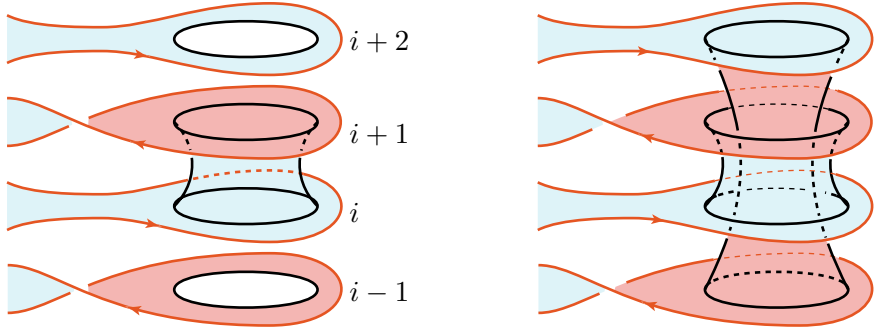
After at most $[\frac{r}{2}]$ such attachments of annuli, all the $\epsilon_i$'s have the same sign, and we get the desired surface $\tilde{F_b}$.
By Proposition~\ref{assumption for a_rho} and Corollary~\ref{def of s} we have the equality (\ref{equation boundary}). 
\end{proof}


\subsection{Construction of the immersed surface ${\check F}_b$ }\label{section check F_b}

We have constructed a surface $\tilde{F_b}$ satisfying the boundary condition (\ref{equation boundary}).
In particular, when $k=0$ we have already obtained an embedded surface $\tilde{F}_b$ whose boundary is $b$.
Define $\Sigma_b := \tilde{F}_b$.

When $k\neq 0$, we construct an immersed surface $\check F_b$ from $\tilde F_b$, by attaching disks about the binding $\gamma'$.

Assume that $k> 0$. 
Proposition~\ref{assumption for a_rho} justifies assuming $a_\rho\geq 0$.
Let ${\mathfrak A}_1, \cdots, {\mathfrak A}_{a_\rho}\subset\tilde{F}_b$ be the ${\mathfrak A}$-annuli whose boundaries contribute to the $a_{\rho}$ copies of $\beta$-circles as in Proposition~\ref{a}.
Recall the number $s=\frac{a_\rho}{k}\geq 0$ defined in Corollary~\ref{def of s}.
Let $u_1, \cdots, u_s \subset A$ be arcs, see Figure~\ref{generators}, disjoint from the Dehn twist circle $\alpha$, such that one end of each $u_i$ sits on the binding $\gamma'$.
Let $\omega_1, \cdots, \omega_s$ be disks, called {\em $\omega$-disks}, obtained by swiping $u_1, \cdots, u_s$  in the open book $(A, D^k)$ so that the center of $\omega_i$ is pierced by $\gamma'$. 
For each $i= 1, \cdots, s$, connect $\omega_i$ smoothly with annuli ${\mathfrak A}_i, {\mathfrak A}_{s+i}, {\mathfrak A}_{2s+i}, \cdots, {\mathfrak A}_{(k-1)s + i}$ by $k$ copies of the twisted band as in Figure~\ref{foliation-1}-(1). 
When $k<0$, attach twisted bands as in Figure~\ref{foliation-1}-(2). 
We have obtained an immersed surface, which we denote by $\check F_b$, see Figure~\ref{clasp-copy}.
\begin{figure}[htpb!]
\begin{center}
\begin{picture}(360, 125)
\put(0,0){\includegraphics{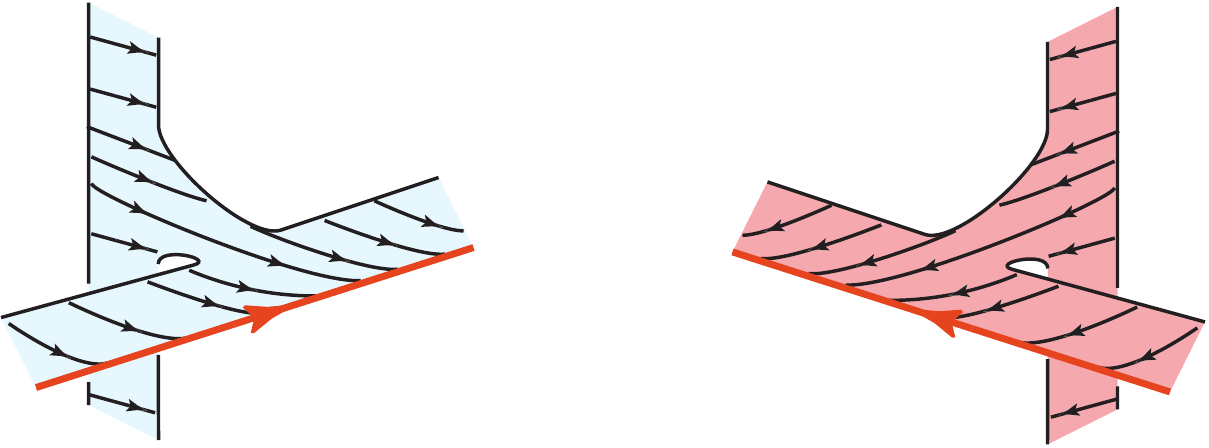}}
\put(-20, 108){\scriptsize(1)}
\put(-20, 95){\scriptsize$a_\rho > 0$}
\put(-20, 80){\scriptsize$k > 0$}
\put(60, 85){\scriptsize twisted}
\put(60, 75){\scriptsize band}
\put(36, 104){\scriptsize disk $\omega_i$}
\put(15, 31){\scriptsize annulus ${\mathfrak A}$}
\put(80, 30){\scriptsize$\rho$}
\put(130, 80){\scriptsize$\beta$}

\put(200, 108){\scriptsize(2)}
\put(200, 95){\scriptsize$a_\rho <0$}
\put(200, 80){\scriptsize$k<0$}
\put(240, 30){\scriptsize$\rho^{-1}$}
\put(350, 30){\scriptsize$-\beta$}

\end{picture}
\caption{An $\omega$-disk and an ${\mathfrak A}$-annulus joined by a twisted brand.}\label{foliation-1}
\end{center}
\end{figure}
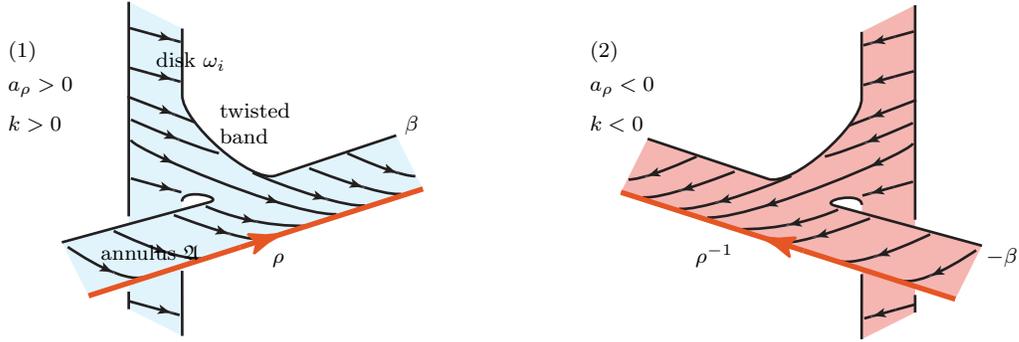

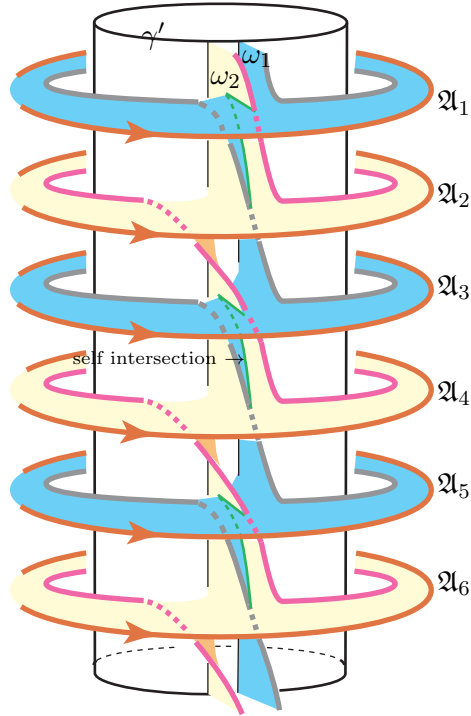
\begin{figure}[htpb!]
\begin{center}
\begin{picture}(160, 270)
\put(0,0){\includegraphics{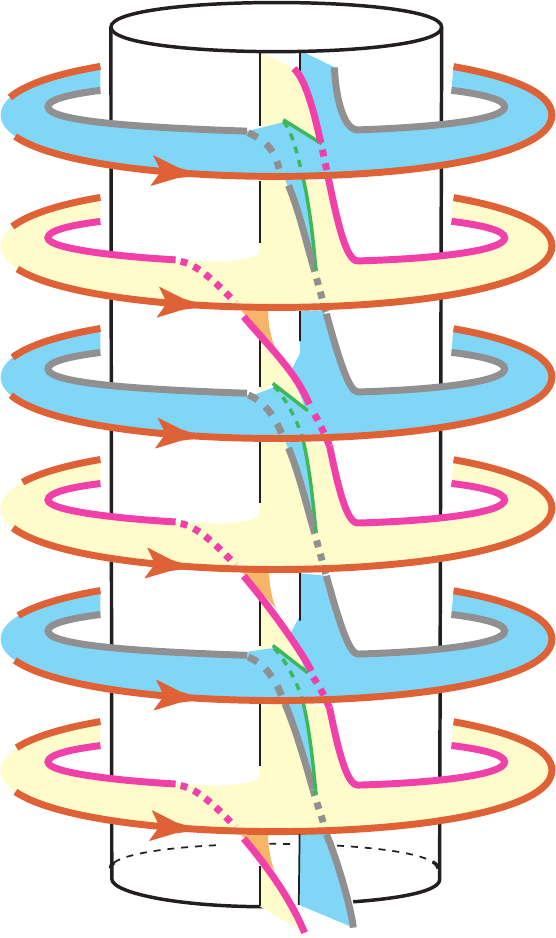}}
\put(88, 248){$\omega_1$}
\put(76, 240){$\omega_2$}
\put(50, 255){$\gamma'$}
\put(162, 230){$\mathfrak A_1$}
\put(162, 195){$\mathfrak A_2$}
\put(162, 160){$\mathfrak A_3$}
\put(162, 120){$\mathfrak A_4$}
\put(24, 134){\tiny self intersection $\rightarrow$}
\put(162, 85){$\mathfrak A_5$}
\put(162, 48){$\mathfrak A_6$}
\end{picture}
\caption{A part of immersed surface $\check F_b$ for $k=3, a_\rho=6, s=2.$ Two $\omega$-disks and six $\mathfrak A$-annuli joined by twisted bands. Self intersections are marked by thin green curves.}\label{clasp-copy}
\end{center}
\end{figure}


\begin{lemma}\label{no singularity on band}

Regardless of the sign of $k$, all the singularities for the characteristic foliation of the attached bands and the $\omega$-disks are positive elliptic.
\end{lemma}

\begin{remark}\label{remark sing points}
The surface $\check F_b$ has $s$ additional positive elliptic points compared to $\tilde F_b$. 
\end{remark}

\begin{proof}

By definition of $\omega$-disk, its characteristic foliation has a single singularity at its center and it is of elliptic type (Figure~\ref{foliation-1}). 
The orientation of $\omega$-disk is induced from that of the $\mathfrak A$-annuli so that the sign of the elliptic point is positive regardless of the sign of $k$. 

In the following, we show that there are no hyperbolic points on the twisted band. 
See Figure~\ref{foliation-1-1}.
Parameterize the twisted band as $[0,1] \times [-1, 1]$. 
Attach the side $\{0\}\times [-1,1]$ of the band to the $\omega$-disk and $\{1\}\times[-1,1]$ to the $\beta$-circle so that the (dashed) line segment $[0,1]\times\{\frac{1}{2}\}$ sit on one page of the open book.
We make the resulting surface smooth near the two attaching sides of the band. 
\begin{figure}[htpb!]
\begin{center}
\begin{picture}(374, 190)
\put(-7,0){\includegraphics{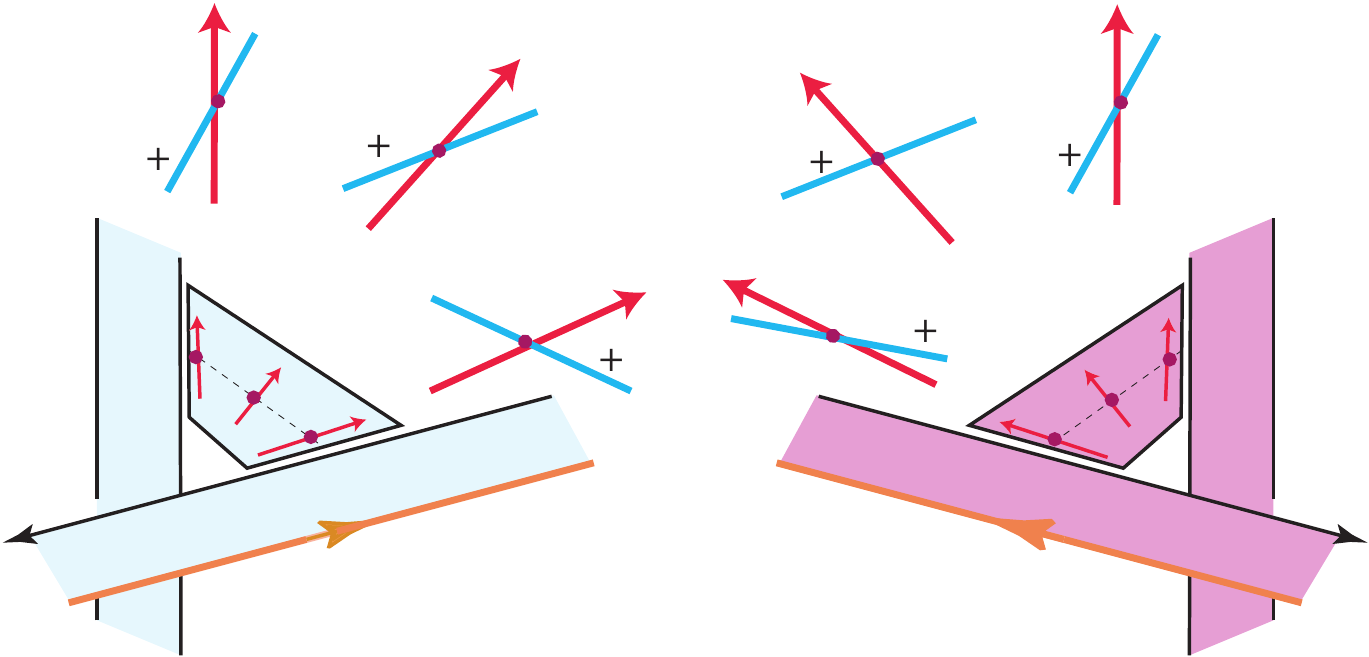}}

\put(140, 80){\scriptsize$p_3$}
\put(80, 70){\scriptsize$p_3$}
\put(170, 110){\scriptsize$v_3$}

\put(60, 85){\scriptsize$p_2$}
\put(115, 155){\scriptsize$p_2$}
\put(130, 170){\scriptsize$v_2$}

\put(35, 90){\scriptsize$p_1$}
\put(45, 160){\scriptsize$p_1$}
\put(43, 180){\scriptsize$v_1$}

\put(40, 30){\scriptsize$\mathfrak A$}
\put(30, 50){\scriptsize$\omega$}
\put(115, 30){\scriptsize$\rho$}
\put(120, 0){\scriptsize$k>0$}
\put(-5, 40){\scriptsize$\beta$}

\put(265, 30){\scriptsize$\rho^{-1}$}
\put(245, 0){\scriptsize$k<0$}
\put(338, 30){\scriptsize$\mathfrak A$}
\put(345, 50){\scriptsize$\omega$}
\put(380, 40){\scriptsize$-\beta$}

\put(240, 95){\scriptsize$p_3$}
\put(255, 139){\scriptsize$p_2$}
\put(320, 155){\scriptsize$p_1$}

\put(305, 175){\scriptsize$v_1$}
\put(235, 160){\scriptsize$v_2$}
\put(215, 110){\scriptsize$v_3$}

\put(290, 70){\scriptsize$p_3$}
\put(310, 85){\scriptsize$p_2$}
\put(320, 90){\scriptsize$p_1$}

\end{picture}
\caption{Proof of Lemma~\ref{no singularity on band}}\label{foliation-1-1}
\end{center}
\end{figure}
Take points on the twisted band $p_1=(0, \frac{1}{2})$, $p_2=(\frac{1}{2}, \frac{1}{2})$, and $p_3=(1, \frac{1}{2})$.
Let $v_i$ be a tangent vector (red arrow) of the band at $p_i$ that is perpendicular to the dashed line $[0,1]\times\{\frac{1}{2}\}$.

When $k>0$, with respect to the page of the open book, $v_1$ is vertical, $v_2$ is slanted $45^\circ$, and $v_3$ is slanted $\epsilon$-degree ($0< \epsilon < 45)$ because the braid $b$ transversely intersects all the pages of the open book positively. 
Next we look at contact planes $\xi_{p_i}$ (light blue line segment) at $p_i$.
In Figure~\ref{foliation-1-1}, the positive side of a contact plane is marked ``$+$".
At each point of the bindings $\gamma, \gamma'$, we may assume that the contact plane is positively perpendicular to the binding. Between the bindings, the contact planes rotate $180^\circ$ counter clockwise along the radial lines.  
Since $p_1$ is close to the binding $\gamma'$, for some $\epsilon_1>0$, $\xi_{p_1}$ is slanted $(90-\epsilon_1)$-degree with respect to the page of the open book. 
While, $\xi_{p_3}$ is slanted $(-\epsilon_3)$-degree for some $\epsilon_3>0$ since $p_3$ is on the circle $\beta$ which is between $\alpha$ and $\gamma$ (Figure~\ref{generators}).
At any point between $p_1$ and $p_3$ on the dashed line, the tangent plane is slanted more than the contact plane. It means that they never coincide. Since the band is a small neighborhood of the dashed line, contact planes are never tangent to the band, hence there are no singularities in the characteristic foliation on the twisted band.

When $k<0$, at $p_3$, the tangent vector $v_3$ is slanted $(180-\epsilon)$-degree and the contact plane $\xi_{p_3}$ is slanted $(-\epsilon_3)$-degree. Braid $b$ is a transverse link so it intersects contact planes positively. Considering that $p_3$ is close to the braid (orange arc), $v_3$ intersects $\xi_{p_3}$ positively, i.e., $\epsilon > \epsilon_3$. 
Therefore, the tangent planes and the contact planes never coincide along the dashed line from $p_1$ to $p_3$, hence there are no singularities in the characteristic foliation on the twisted band.
\end{proof}

\subsection{Construction of the immersed surface ${\hat F}_b$ }\label{section hat F_b}

Let $\tau$ be a closed braid in $(A, D^k)$ of braid index $=1$. 
The immersed surface $\check F_b$ constructed in Section~\ref{section check F_b} has boundary $[b] + s[\tau+k\beta]$ in $H_1(M_{(A, D^k)}, \Z)$.
Each closed curve representing $-[\tau+k\beta]$ bounds a disk about the binding $\gamma$. We call it a  {\em $\mathcal D$-disk}, see Figure~\ref{D-disk}. 
\begin{figure}[htpb!]
\begin{center}
\begin{picture}(252, 260)
\put(0,0){\includegraphics{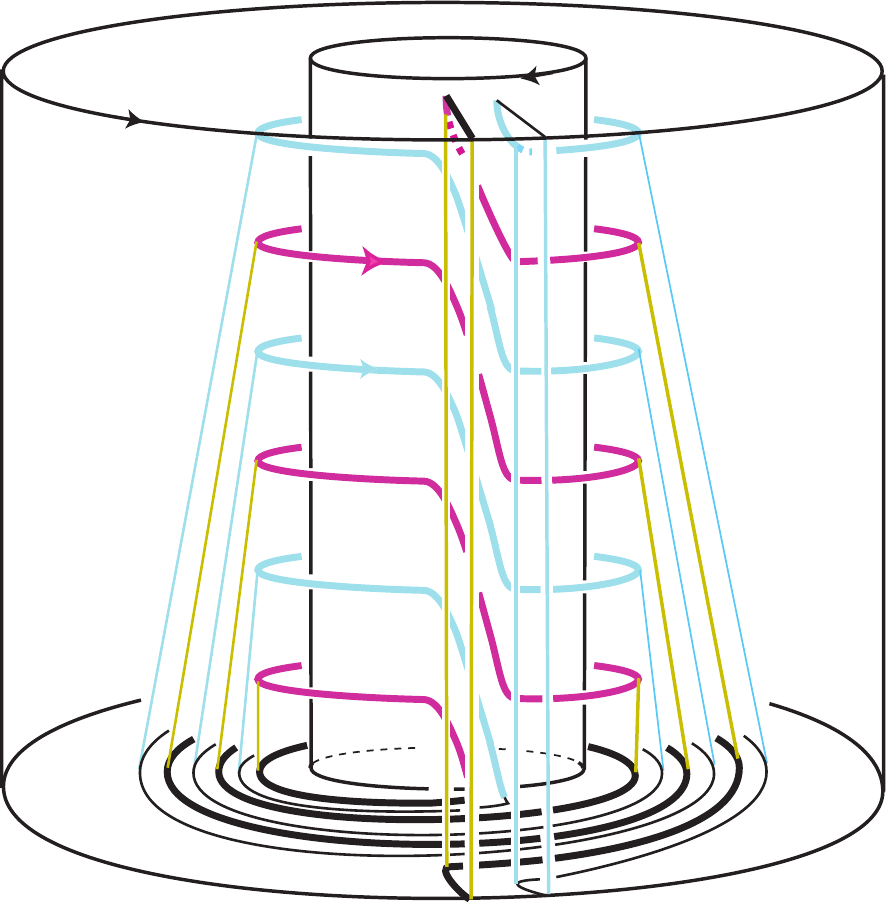}}
\put(36, 216){$\gamma$}
\put(108, 240){$\gamma'$}
\end{picture}
\caption{$\mathcal D$-disks}\label{D-disk}
\end{center}
\end{figure}
There, the spirals in the bottom annulus page are identified, via the Dehn twist $D^k$, with the straight line segments in the top annulus page.
There are $s$-copies of $\mathcal D$-disk and they are disjoint from each other.  
Since $\mathcal D$-disks are nearly `vertical' as in Figure~\ref{D-disk}, the tangent planes and the contact planes, which rotate $180^\circ$ counter clockwise along the radial lines from $\gamma$ to $\gamma'$,  intersect transversely. This means that the characteristic foliation of each $\mathcal D$-disk has a single singularity, which occurs at the intersection point with $\gamma$ and whose type is elliptic. 
The orientations of the $\mathcal D$-disks are compatible with those of the boundaries so that the $\mathcal D$-disks and $\gamma$ intersect negatively. 
Therefore the signs of the elliptic points are negative.

\begin{definition}
We construct an immersed surface $\hat F_b$ by glueing the $\mathcal D$-disks and $\check F_b$ along the $s$ copies of the $\tau + k\beta$ curve. 
\end{definition}

\begin{remark}\label{sign of elliptic points}
This $\hat{F}_b$ has $s$ additional negative elliptic singularities given by the $\mathcal D$-disks compared to the surface $\check F_b$.
\end{remark}




\subsection{Resolution of singularities}\label{section resolution}

We start this section by defining three types of intersection of surfaces; branch, clasp, and ribbon, then study resolution of self-intersections.

\begin{definition}\label{def of branch}
Let $\Sigma$ be an immersed oriented surfaces with $\partial \Sigma = K$ given by the immersion $i:\tilde{\Sigma}\to \Sigma$. Let $l\subset \Sigma$ be a simple arc where $\Sigma$ intersects itself, and denote by 
$p$ and  $q$ the endpoints of $l$.

$\bullet$
If $p$ is  sitting on $K$, and $q$ is a branch point of a neighborhood Riemann surface, see Figure~\ref{branch}-(1),  then we call $l$ a {\em branch} intersection.
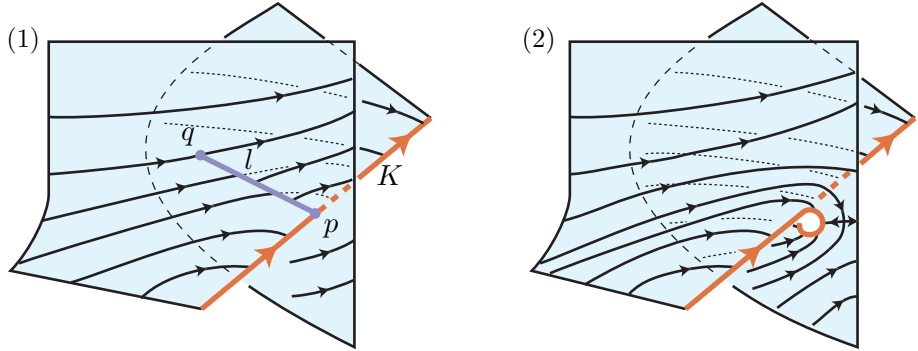
\begin{figure}[htpb!]
\begin{center}
\begin{picture}(340, 130)
\put(0,0){\includegraphics{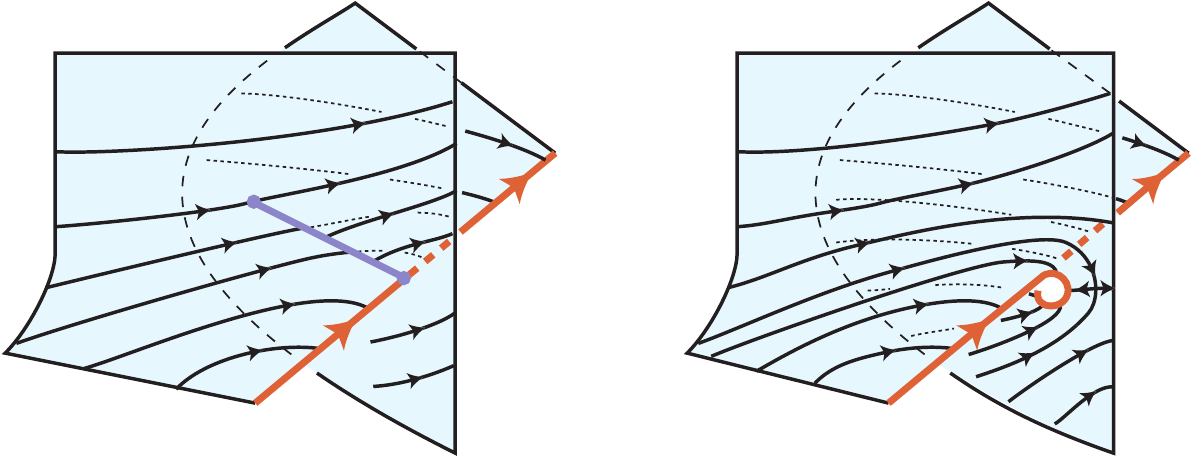}}
\put(90, 69){$l$}
\put(0, 115){\small $(1)$}
\put(195, 115){\small $(2)$}
\put(140, 63){$K$}
\put(120, 45){$p$}
\put(66, 80){$q$}
\end{picture}
\caption{(1) A negative branch intersection $l$, and  (2) its resolution.   }\label{branch}
\end{center}
\end{figure}

Next assume that the preimage of $l$, $i^{-1}(l)\subset \tilde{\Sigma}$, consists of two arcs, say $\tilde l_1, \tilde l_2$. Denote the end points of $\tilde l_i$ by $\tilde{p_i}$ and $\tilde{q_i}$ for $i=1,2$.

$\bullet$
If $\tilde p_1, \tilde q_2 \in \partial \Sigma$ and $\tilde p_2, \tilde q_1 \in{\rm Int}(\tilde \Sigma)$ then we call the intersection a  {\em clasp} intersection.  
A local picture of $l$ is the left sketch of Figure~\ref{clasp-ribbon}.

$\bullet$
If $\tilde p_1, \tilde q_1 \in \partial \Sigma$ and $\tilde p_2, \tilde q_2 \in{\rm Int}(\tilde \Sigma)$ then we call the intersection a {\em ribbon}  intersection. See the right sketch of Figure~\ref{clasp-ribbon}. 
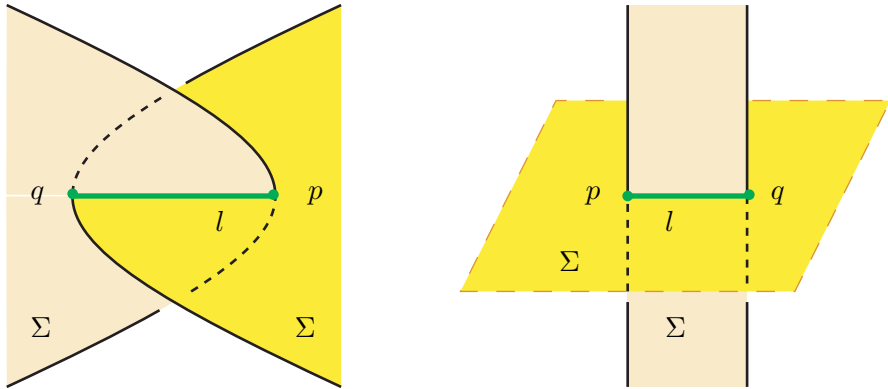
\begin{figure}[htpb!]
\begin{center}
\begin{picture}(332, 144)
\put(0,0){\includegraphics{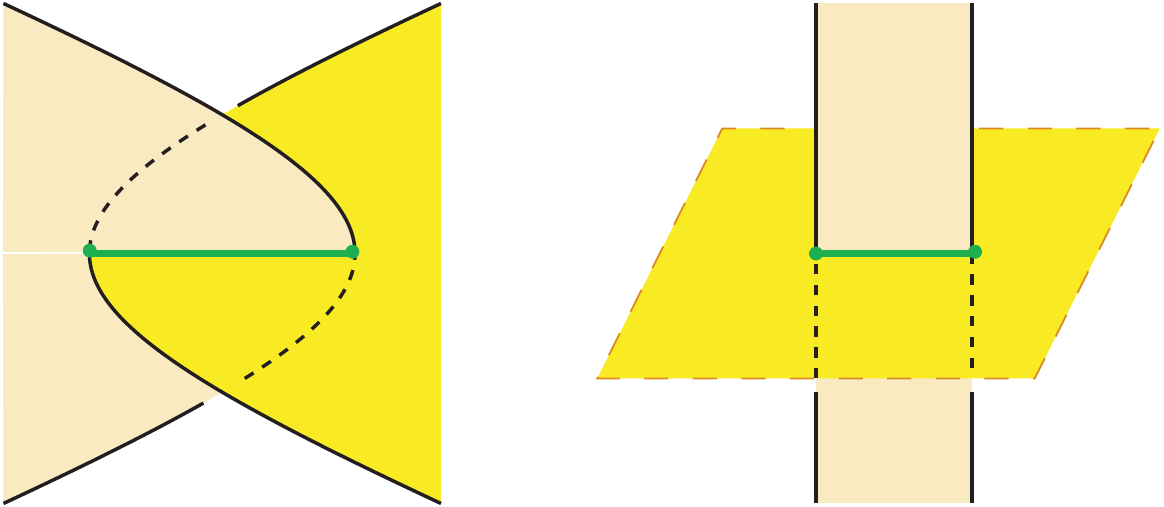}}
\put(80, 60){$l$}
\put(10, 72){$q$}
\put(115, 72){$p$}
\put(10, 20){$\Sigma$}
\put(110, 20){$\Sigma$}

\put(220, 72){$p$}
\put(290, 72){$q$}
\put(250, 60){$l$}
\put(250, 20){$\Sigma$}
\put(210, 45){$\Sigma$}

\end{picture}
\caption{(Left) A clasp intersection. (Right) A ribbon intersection.   }\label{clasp-ribbon}
\end{center}
\end{figure}
\end{definition}

\begin{example}\label{example of singularities}
See Figure~\ref{clasp-branch}.
The immersed surface, $\hat{F}_b$, has: 

\noindent $\bullet$
$|a_\rho|$ branches formed by $\mathfrak A$-annuli and $\mathcal D$-disks. %
\begin{figure}[htpb!]
\begin{center}
\begin{picture}(380, 446)
\put(0,0){\includegraphics{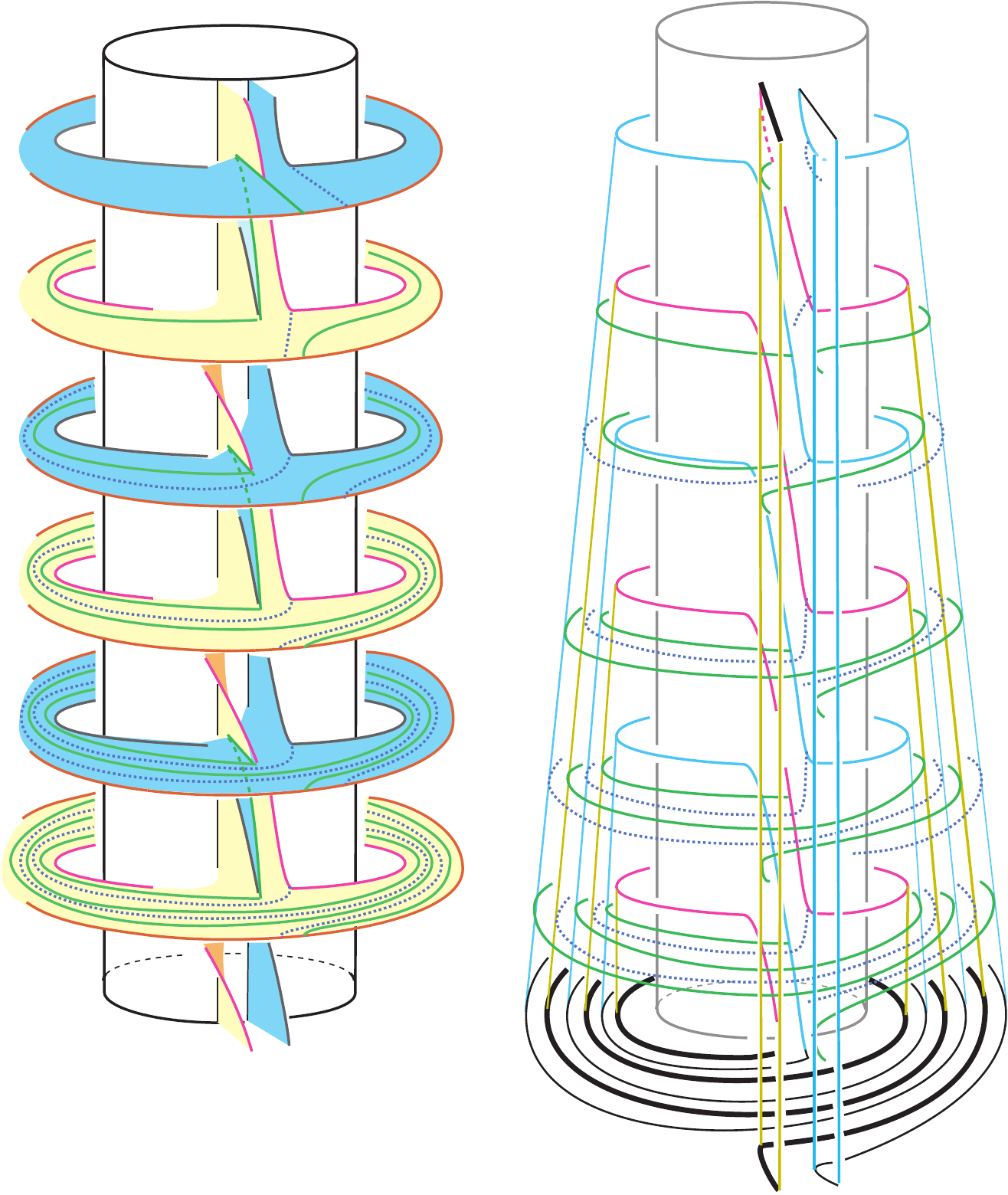}}
\put(110, 380){\tiny{branch}}
\put(70, 330){\tiny{clasp}}
\end{picture}
\caption{Clasp (green) and branch (dashed purple) intersections on $\hat F_b$. }\label{clasp-branch}
\end{center}
\end{figure}

\noindent $\bullet$
$|k|\binom{s}{2}=\frac{1}{2}|a_{\rho}|(s-1)$ clasp intersections when $s>1$. Recall that $\hat{F}_b$  is obtained by attaching $s$ copies of $\mathcal D$-disk about the binding $\gamma$. Each pair among these $s$ disks interacts as in  Figure~\ref{clasp-branch} giving rise to $|k|$ clasp intersections. 
When $s=1, 0$, there are no clasps.

\noindent $\bullet$
several ribbon intersections of $\mathcal D$-disks and the (nested) vertical  annuli of Figure~\ref{canceling_rho}.

\end{example}

In Section~\ref{section Sigma_b}, we resolve these self-intersections to obtain an embedded surface $\Sigma_b$.

In the following, we assume that $K$ is a transverse knot in a contact manifold $(M, \xi)$ and $\Sigma$ an immersed oriented surface with $\partial \Sigma = K$. Also,
we assume that (i) the self-intersection set of $\Sigma$ consists of ribbon, clasp, or branch intersections; (ii) the characteristic foliation ${\mathcal F}_\Sigma$ is of Morse-Smale type.

Let $l \subset \Sigma$ be a self-intersection arc.
Near a point $x \in {\rm Int}(l)$, $\Sigma$ intersects itself transversely as in Figure~\ref{resolution}-(1).
Let $F_i \subset \Sigma$ ($i=1, 2, 3, 4$) be surfaces meeting at $l$. The orientation of $F_i$ is induced from that of $\Sigma$.
Resolve the singularity $l$ by cutting $\Sigma$ out along $l$ and re-gluing $F_1, F_2$ along
$l$ and $F_3, F_4$ along $l$ so that the orientations of the surfaces agree. See Figure~\ref{resolution}-(2). Call the new surface $\Sigma'$.
\begin{figure}[htpb!]
\begin{center}
\begin{picture}(280, 80)
\put(0,0){\includegraphics{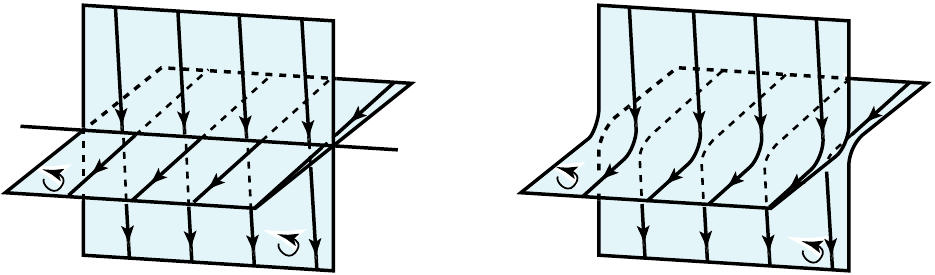}}
\put(-5, 40){$l$}
\put(8, 70){$F_1$}
\put(-13, 20){$F_2$}
\put(100, 5){$F_3$}
\put(110, 60){$F_4$}
\put(-20, 70){(1)}
\put(155, 70){(2)}
\end{picture}
\caption{(1) Immersed surface $\Sigma$. (2) New surface $\Sigma'$ after resolution of singularity $l$.  }\label{resolution}
\end{center}
\end{figure}
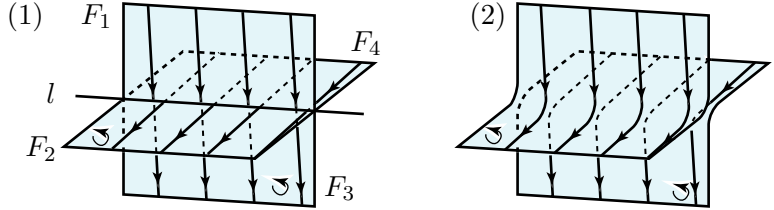

We orient the leaves of the characteristic foliation following \cite[page 80]{OS}: For $p\in\Sigma$ a nonsingular point of a leaf $L$ of the foliation, let $\vec{n}\in T_p\Sigma$ be a positive normal vector to $\xi_p$. 
We choose a vector $\vec{v} \in T_pL$ so that $(\vec{v}, \vec{n})$ is a positive basis for $T_p\Sigma$.
This vector field $\vec{v}$ determines the orientation of the characteristic foliation.

We observe that if both ${\mathcal F}_{F_1}$ and ${\mathcal F}_{F_2}$ transversely intersect the line $l$ (Figure~\ref{resolution}-(1)), then the orientations of ${\mathcal F}_{F_1}$ and ${\mathcal F}_{F_2}$ agree at $l$. 
Hence, after the cut and glue operation, the new characteristic foliation $\mathcal F_{\Sigma'}$ is obtained by smoothly connecting the old ${\mathcal F}_{F_1}$ and ${\mathcal F}_{F_2}$, and also ${\mathcal F}_{F_3}$ and ${\mathcal F}_{F_4}$.
See Figure~\ref{resolution}-(2).

Near the endpoints of $l$, this resolution creates new hyperbolic
points and ${\mathcal F}_{\Sigma'}$ can be made into Morse-Smale type. See Figure~\ref{branch}-(2) and Figure~\ref{resolution-2}-(2).
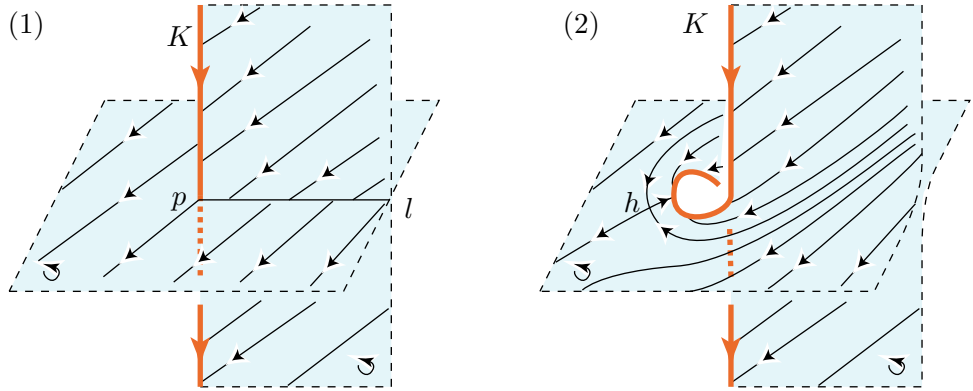
\begin{figure}[htpb!]
\begin{center}
\begin{picture}(370, 144)
\put(0,0){\includegraphics{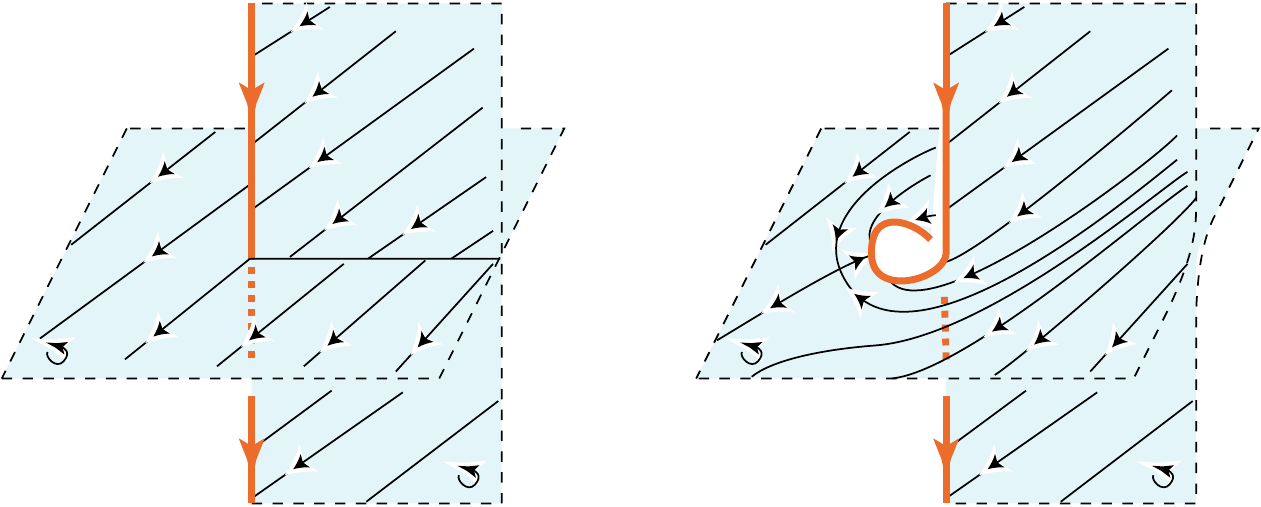}}
\put(0, 135){(1)}
\put(210, 135){(2)}
\put(150, 65){$l$}
\put(62, 70){$p$}
\put(60, 130){$K$}
\put(233, 67){$h$}
\put(255, 135){$K$}
\end{picture}
\caption{(1) A negative intersection $p$. (2) Creation of a negative hyperbolic singularity $h$ by resolving singular arc $l$.}\label{resolution-2}
\end{center}
\end{figure}
The signs of the new hyperbolic points are determined in the following way:

\begin{proposition}\label{prop-sign}
Suppose that $p \in \partial l \cap K$ and both ${\mathcal F}_{F_1}$ and ${\mathcal F}_{F_2}$ are transversely intersecting with $l$.
If $p$ is a positive (negative) transverse intersection of $K$ and $\Sigma$, then the new hyperbolic point has positive (negative) sign.
\end{proposition}

\begin{proof}
Assume that $p$ is a negative intersection, as depicted in
Figure~\ref{resolution-2}-(1). We introduce an $(x,y,z)$-coordinate
system for a small neighborhood $N$ of $p$: Identify $p$ with
$(0,0,0)$, and identify $-K$ with the $z$-axis. Regard the surface
which $K$ penetrates as the $xy$-plane. Since $K$ is a transverse
knot, it transverses the contact $2$-planes positively. Thus at a
point $r \in K\cap N$ the positive normal vector $\vec{n}_r$ to the
contact plane $\xi_r$ has a negative $z$-component, i.e., $\vec{n}_r
\cdot (0,0,1) < 0$.
We may assume that contact planes are almost parallel to each other
in $N$. Therefore, at the new hyperbolic point $h\in N$ we have
$T_h\Sigma = -\xi_h$. This means that $h$ is a negative hyperbolic
point.
\end{proof}

Since the two end points of a ribbon (resp. clasp) singularity have the same  sign (resp. opposite signs), it follows that:
\begin{corollary}
\label{cor-sign}
(1) The resolution of a ribbon singularity creates one positive and one negative hyperbolic points. \\
(2) The resolution of a clasp singularity creates two hyperbolic points of the {\em same} sign. \\
(3) The resolution of a branch singularity creates one hyperbolic point. See Figure~\ref{branch}.
\end{corollary}

It makes sense to define {\em the sign} for clasp and branch arcs:
\begin{definition}\label{def sign of singularity}
\begin{enumerate}
\item If both end points of a clasp arc are positive (negative) intersections of $K$ and $\Sigma$, then we say the {\em sign} of the clasp is {\em positive} ({\em negative}).
\item If the end point $p=l \cap K$ of a branch arc is a positive (negative) intersection, then we say the {\em sign} of the branch arc is {\em positive} ({\em negative}).
\end{enumerate}
\end{definition}


\subsection{Construction of the embedded surface $\Sigma_b$ }\label{section Sigma_b}

In Section~\ref{section check F_b} we have defined a Seifert surface $\Sigma_b$ for the case $k=0$. 

When $k\neq 0$, Example~\ref{example of singularities} shows that the immersed surface $\hat{F}_b$ has branch, clasp, and ribbon intersections.
In this section, we construct an embedded surface $\Sigma_b$ by resolving these singularities.

When $k>0$, as shown in Figure~\ref{foliation-2}-(1,2), we can make all the branch, ribbon and clasp arcs transverse to the characteristic foliation $\mathcal{F}_{{\hat F}_b}$.
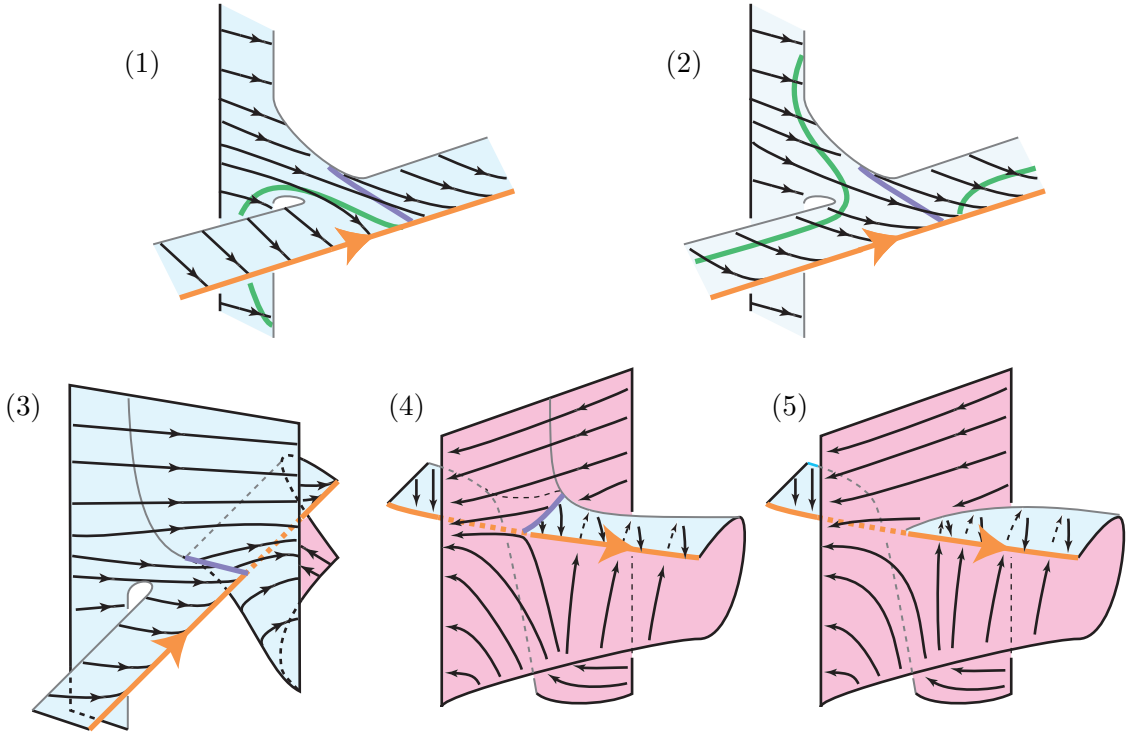
\begin{figure}[htpb!]
\begin{center}
\begin{picture}(400, 275)
\put(0,0){\includegraphics{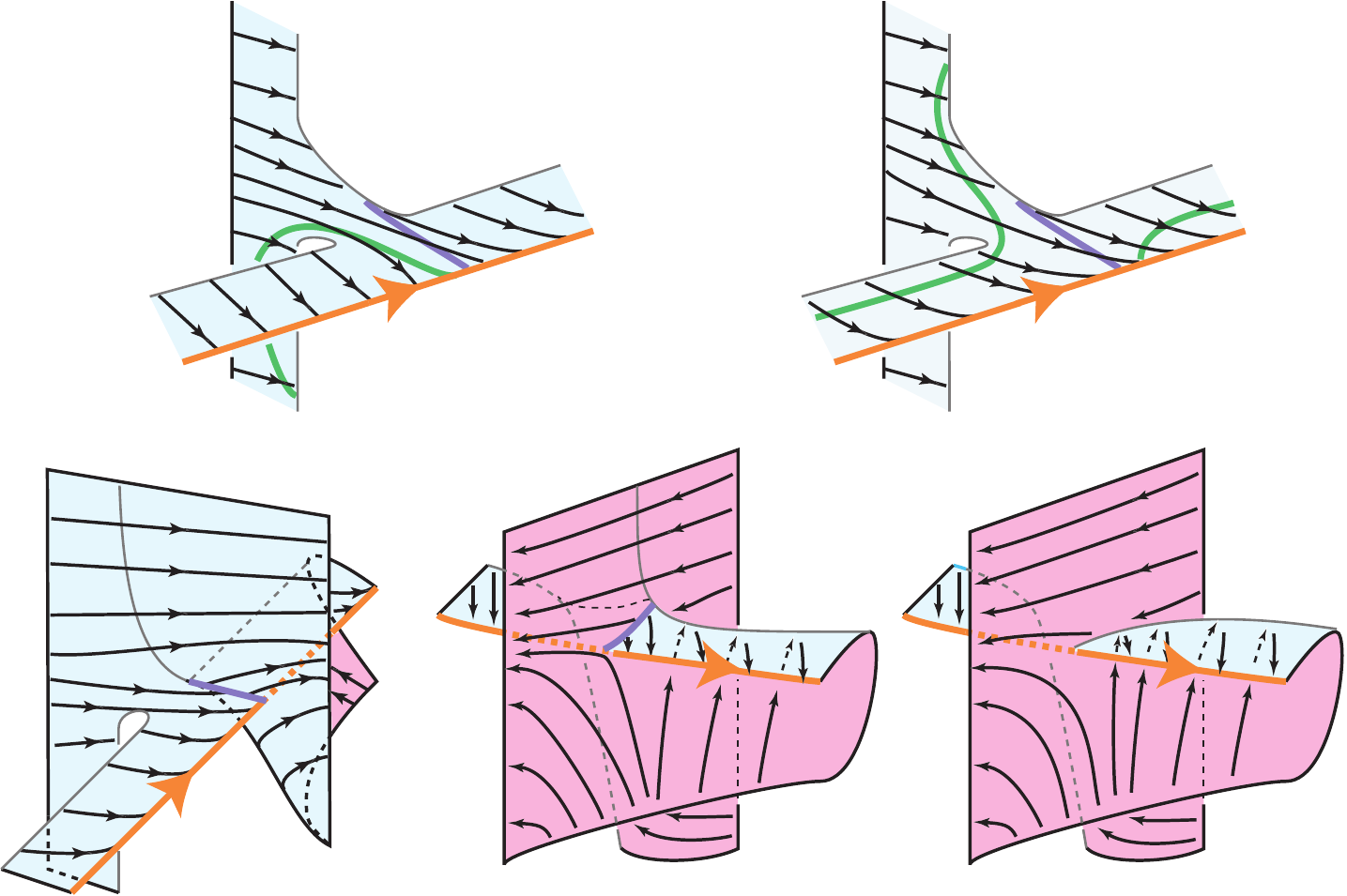}}
\put(35, 250){(1)}
\put(240, 250){(2)}
\put(-10, 120){(3)}
\put(135, 120){(4)}
\put(280, 120){(5)}
\end{picture}
\caption{Clasp (light green) and branch (dark purple) intersections are transverse to characteristic foliations (top). Characteristic foliation near a branch singularity and its resolution (bottom).}\label{foliation-2}
\end{center}
\end{figure}
We apply the argument in Section~\ref{section resolution} and construct an embedded surface $\Sigma_b$.
Since all the signs of the branch and clasp arcs are negative, Example~\ref{example of singularities} and Corollary~\ref{cor-sign} imply that, when $s>0$, the resolution of these self-intersections creates, in total, $a_\rho + 2(\frac{1}{2}a_{\rho}(s-1))=a_{\rho}s$ negative hyperbolic singularities. When $s=0$ there are no branch or clasp intersections, so no new hyperbolic points are created.

When $k<0$, a similar argument holds.


\section{ The Self linking number }\label{section sl}

We finally compute the self-linking number $sl(b, [\Sigma_b])$ of $b$ relative to the embedded surface $\Sigma_b$.

\begin{theorem}\label{sl-thm}
Let $b$ be a null-homologous transverse braid in the open book decomposition $(A, D^k)$ satisfying Assumption~\ref{position-assumption}. There exists a Seifert surface $\Sigma_b$ of $b$ which satisfies the formula:
$$
sl(b, [\Sigma_b])=-n+a_{\sigma}+a_{\rho}(1-s).
$$
In particular, when $k\ne 0$, $sl(b, [\Sigma_b])$ does not depend on the choice of Seifert surface and hence we have the following general formula:
$$sl(b) = -n+a_{\sigma}+a_{\rho}(1-s).$$
\end{theorem}

\begin{remark}
\begin{itemize}
\item
When $k\neq 0$, our manifold $M_{(A, D^k)}$ is a lens space (Claim~\ref{ob}) which has $H_2(L(k,q), \Z)=0$. Therefore, the self linking number does not depend on choice of a Seifert surface and we can denote $sl(b, [\Sigma_b])$ simply by $sl(b)$.

\item
The null-homologous condition ensures $a_\rho=0$ when $k=0$ (Corollary~\ref{def of s}).

\item
When $a_\rho=0$ we exactly obtain the Bennequin's formula (\ref{Bennequin-eq}).
\end{itemize}
\end{remark}

\begin{proof}[Proof of Theorem~\ref{sl-thm}]

Let $\Sigma_b$ be the surface constructed in Section~\ref{section Sigma_b}.

It is known (see \cite{E1} for example) that
\begin{equation}\label{sl-formula}
sl(b, [\Sigma_b])=-(e^+-e^-)+(h^+-h^-),
\end{equation}
where $e^+$ ($e^-$) and $h^+$ ($h^-$) represent the number of positive (negative) elliptic and positive (negative) hyperbolic singularities of the characteristic foliation ${\mathcal F}_{\Sigma_b}$ on $\Sigma_b$.
Let $h^+_\sigma$ ($h^-_\sigma$) be the number of $\sigma_i$'s ($\sigma_i^{-1}$'s) which appear in the braid word for $b$.
Then we have $a_{\sigma}=h^+_{\sigma}-h^-_{\sigma}$, the
sign count of hyperbolic singularities on $\Sigma_b$ given by the
bands joining $\delta$-disks as in Figure~\ref{surface1}-(1).

Based on Remarks~\ref{remark sing points}, \ref{sign of elliptic points} and Section~\ref{section Sigma_b}, we summarize the count of singularities: 
\begin{eqnarray*}
e^+ &=& (n; \small\mbox{ on $\delta$-disks}) + (s; \small\mbox{  on $\omega$-disks}) \\
e^- &=& (s; \small\mbox{ on $\mathcal D$-disks) } \\
h^+ &=& (h^+_\sigma; \small\mbox{ on $+$bands between $\delta$-disks}) + (a_{\rho}; \small\mbox{ on bands between $\delta_n$ and ${\mathfrak A}$-annuli}) \\
h^- &=& (h^-_\sigma;  \small\mbox{ on $-$bands between $\delta$-disks}) + (a_{\rho} s;
\small\mbox{ by resolution of branches, clasps, ribbons})
\end{eqnarray*}
By (\ref{sl-formula}) we obtain the desired formula.
\end{proof}

The Bennequin-Eliashberg inequality~\cite{Ben, El} states that contact structure $(M, \xi)$ is tight if and only if
\begin{equation}\label{bennequin-inequality}
sl(K, [\Sigma]) \leq - \chi(\Sigma)
\end{equation}
for any $(K, \Sigma)$ a null-homologous transverse knot and its Seifert surface.

\begin{corollary}\label{cor of sl-formula}

The contact structure $(M_{(A, D^k)}, \xi_k)$ is tight if and only if for any braid $b \subset (A, D^k)$ inequality $sl(b) \leq - \chi(\Sigma_b)$ holds.

\end{corollary}

\begin{proof}[Proof of Corollary~\ref{cor of sl-formula}]

As $\chi(\Sigma_b)= (e^+ + e^-) - (h^+ + h^-),$
the inequality $sl(b) \leq - \chi(\Sigma_b)$ is equivalent to $0 \leq h^- - e^- = h^-_\sigma + s(a_\rho -1)$.
Claim~\ref{goodman} states that $(M_{(A, D^k)}, \xi_k)$ is tight if and only if $k\geq 0$.

When $k\geq 0$, we have $s(a_\rho -1) = s(ks-1) \geq 0$, thus $h^- - e^-  \geq 0$.

When $k<0$, we have $s(a_\rho -1) < 0$ by Corollary~\ref{def of s}. Therefore, there exists $b$ for which  $h^- - e^-  <0$.
\end{proof}

\begin{remark}\label{remark-ineq}
It is interesting to note that the Bennequin-Eliashberg inequality is {\em not} satisfied for the immersed surface $\hat{F}_b$ even for the tight cases.
\end{remark}

Next, we study behavior of $sl(b, [\Sigma_b])$ under braid stabilizations.

Let $b$ be a null-homologous braid in the open book $(A, D^k)$.
For $\epsilon\in\{+, -\}$ let $b_\epsilon^\gamma$ (resp. $b_\epsilon^{\gamma'})$ denote the braid obtained from $b$ after an $\epsilon$-stabilization about the binding $\gamma$ (resp. $\gamma')$.
By Theorem~\ref{elena2}, braids $b, b_+^\gamma$ and $b_+^{\gamma'}$ are transversely isotopic regardless of choice of stabilization arc $a$.
Etnyre's \cite[Theorem 3.8]{E0} implies that if $b$ is a one component link, then a negative stabilization is unique up to transverse isotopy, regardless of choice of arc $a$, thus $b_-^\gamma=b_-^{\gamma'}$.

\begin{corollary}\label{thm-stab}
we have:
\begin{eqnarray}
sl(b,[\Sigma_b])= sl(b_+^{\gamma},[\Sigma_{b_+^{\gamma}}])=
sl(b_-^{\gamma},[\Sigma_{b_-^{\gamma}}])+2, \label{g}\\
sl(b,[\Sigma_b])= sl(b_+^{\gamma'},[\Sigma_{b_+^{\gamma'}}])=
sl(b_-^{\gamma'},[\Sigma_{b_-^{\gamma'}}])+2. \label{g'}
\end{eqnarray}
\end{corollary}

\begin{proof}[Proof of Corollary~\ref{thm-stab}]

A positive (negative) stabilization about $\gamma$ changes $n \to n+1$ and $a_\sigma \to a_\sigma +1$ ($a_\sigma \to a_\sigma -1$). Applying Theorem~\ref{sl-thm}, we get (\ref{g}).

For (\ref{g'}), as we have seen in the proof of Proposition~\ref{assumption for a_rho}, by a positive (negative) braid stabilization, we have the following data change:
$$
n  \to n+1, \ \
a_{\sigma}  \to   a_{\sigma}+1 + 2 a_\rho, \ \
s  \to  s+1, \ \
a_{\rho}  \to  a_{\rho}+k,
$$
$$
(n  \to n+1, \ \
a_{\sigma}  \to   a_{\sigma}-1 + 2 a_\rho, \ \
s  \to  s+1, \ \
a_{\rho}  \to  a_{\rho}+k).
$$
Applying Theorem~\ref{sl-thm}, we have
\begin{eqnarray*}
sl(b_+^{\gamma'}, [\Sigma_{b_+^{\gamma'}}])
&=& -(n+1)+(a_\sigma + 1 + 2a_\rho) + (a_{\rho}+k) (1- (s+1)) \\
&=& -n+a_{\sigma}+a_{\rho}(1-s) \\
&=& sl(b, [\Sigma_b]).
\end{eqnarray*}
A similar computation leads to $sl(b_-^{\gamma'}, [\Sigma_{b_-^{\gamma'}}])+2 = sl(b, [\Sigma_b]).$
\end{proof}


\section{Seifert fibered manifolds}\label{section-seifert}

Let $S$ be an oriented pair of pants with boundary circles $\gamma_i$ for $i=1, 2, 3$. See Figure~\ref{3-punk-disk}.
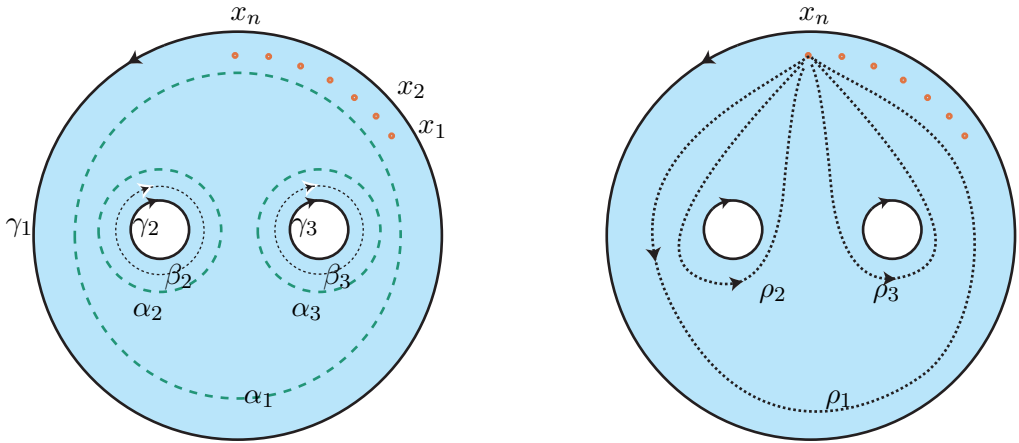
\begin{figure}[htpb!]
\begin{center}
\begin{picture}(375, 160)
\put(0,0){\includegraphics{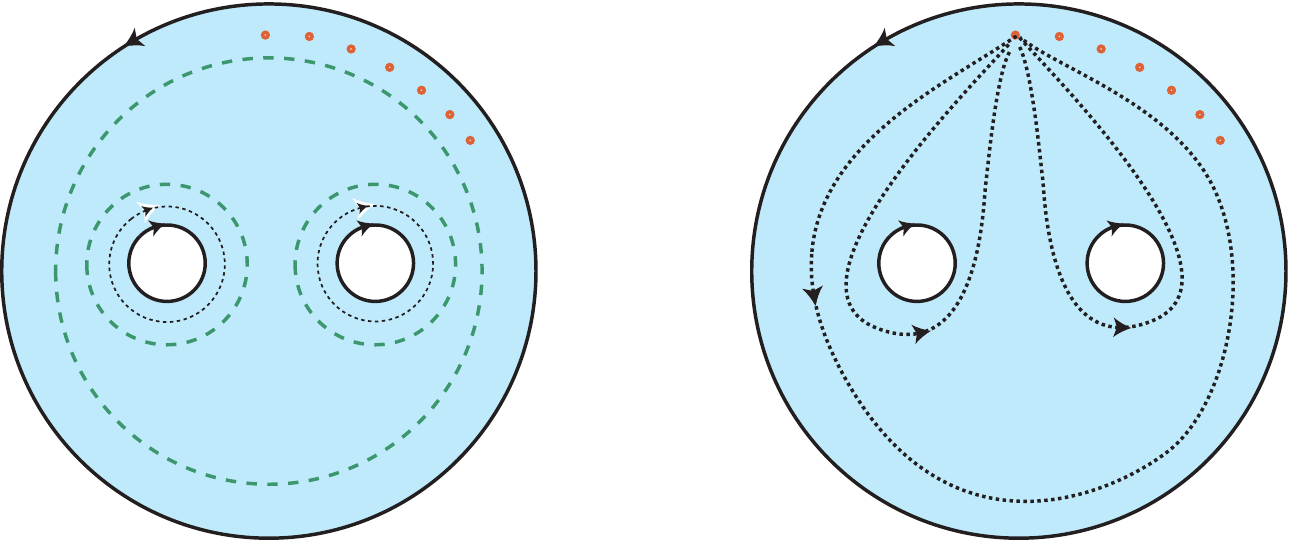}}
\put(-10, 80){$\gamma_1$}
\put(38, 80){$\gamma_2$}
\put(98, 80){$\gamma_3$}
\put(80, 15){$\alpha_1$}
\put(50, 60){$\beta_2$}
\put(110, 60){$\beta_3$}
\put(38, 48){$\alpha_2$}
\put(98, 48){$\alpha_3$}
\put(146, 117){$x_1$}
\put(138, 132){$x_2$}
\put(75, 160){$x_n$}
\put(290, 160){$x_n$}
\put(300, 15){$\rho_1$}
\put(275, 55){$\rho_2$}
\put(318, 55){$\rho_3$}
\end{picture}
\caption{A pair of pants $S$.}\label{3-punk-disk}
\end{center}
\end{figure}
Let $\alpha_i$ be circles parallel to $\gamma_i$. Denote the positive Dehn twist about $\alpha_i$ by $D_i$. Let $k_i$ be an integer, $i=1,2,3$. In this section we study closed braids in the open book decomposition $(S, D_1^{k_1} \circ D_2^{k_2} \circ D_3^{k_3})$.
The corresponding manifold, which we denote by $M_{k_1, k_2, k_3}(=M)$, is a Seifert fibered space over the orbifold of signature $(0, k_1, k_2, k_3)$. 
A similar argument as in the proof of Claim~\ref{ob} tells that  $M_{k_1, k_2, k_3}$ has surgery descriptions as in Figure~\ref{surg-diagram}. 
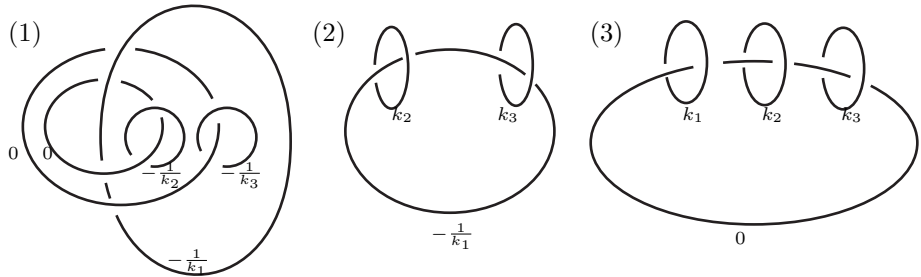
\begin{figure}[htpb!]
\begin{center}
\begin{picture}(340, 100)
\put(0,0){\includegraphics{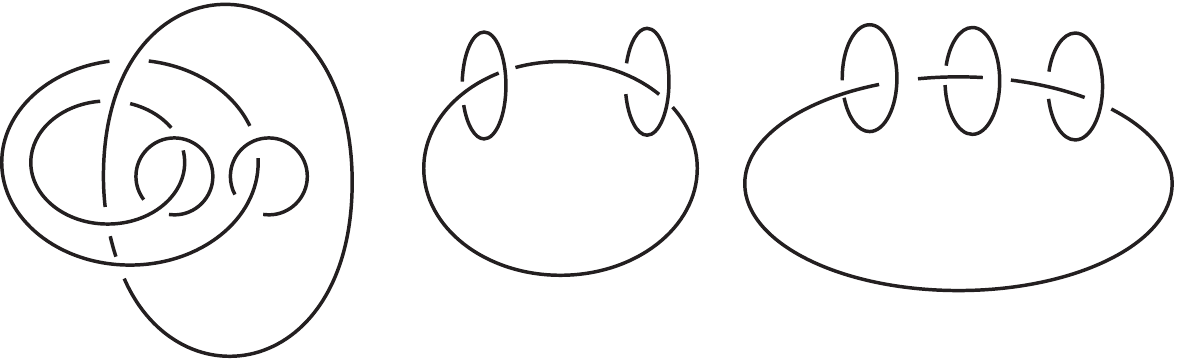}}
\put(-5, 45){\tiny$0$}
\put(8, 45){\tiny$0$}
\put(45, 38){\tiny $-\frac{1}{k_2}$}
\put(75, 38){\tiny $-\frac{1}{k_3}$}
\put(55, 5){\tiny $-\frac{1}{k_1}$}
\put(155, 15){\tiny$-\frac{1}{k_1}$}
\put(140, 60){\tiny$k_2$}
\put(180, 60){\tiny$k_3$}
\put(270, 13){\tiny$0$}
\put(250, 60){\tiny$k_1$}
\put(280, 60){\tiny$k_2$}
\put(310, 60){\tiny$k_3$}
\put(-5, 90){\small(1)}
\put(110, 90){\small(2)}
\put(215, 90){\small(3)}
\end{picture}
\caption{Surgery descriptions for $M_{k_1, k_2, k_3}$.}\label{surg-diagram}
\end{center}
\end{figure}
In Sketch (1), the two circles with slope $0$ represent the unlinked unknots through the holes of $S$, $\gamma_2$ and $\gamma_3$ (cf. the unknot $U$ in the proof of Claim~\ref{ob}). Slam-dunk operations are applied in the passage Sketch $(1)\to (2) \to (3)$.
Etnyre-Ozbagci \cite[page 3136]{EO} implies:
\begin{equation}
H_1(M_{k_1, k_2, k_3};\Z) = \langle c_2, c_3\ | \ (k_1 + k_2) c_2 + k_1 c_3= k_1 c_2 + (k_1+k_3) c_3 = 0 \ \rangle, \label{prop-homology} 
\end{equation}
where $c_j \in \pi_1(S)$ is the standard generator corresponding to the boundaries $\gamma_j$ of $S$.
We remark that $H_2(M;\Z)$ is non-trivial in general, which is distinct from the lens spaces.

\begin{proposition}
Let $\xi_{k_1, k_2, k_3}$ denote the contact structure for $M_{k_1, k_2, k_3}$ compatible with the open book $(S, D_1^{k_1} \circ D_2^{k_2} \circ D_3^{k_3})$ via the Giroux correspondence \cite{G}.
Then $\xi_{k_1, k_2, k_3}$ is tight if and only if $k_1, k_2, k_3 \geq 0$.
\end{proposition}

\begin{proof}

If $k_1=k_2=k_3=0$, then $M_{0,0,0}=(S^1\times S^2) \# (S^1 \times S^2)$. Etnyre-Honda explain in the proof of \cite[Lemma 3.2]{EH} that $\xi_{0,0,0}$ is Stein fillable, hence tight. 
In fact, it is the unique tight structure, which is due to Eliashberg~\cite{El}.

If $k_1, k_2, k_3\geq 0$ and $(k_1, k_2, k_3)\neq(0,0,0)$, Etnyre-Honda \cite[Lemma 3.2]{EH} tells that $\xi_{k_1, k_2, k_3}$ is tight.

If one of $k_i$ is negative, say $k_1<0$, then a properly  embedded boundary {\em non}-parallel essential arc whose both ends sit on $\gamma_1$ is a sobering arc, see \cite[Definition 3.2]{Go}. Thus Goodman's \cite[Theorem 1.2]{Go} implies that $\xi_{k_1, k_2, k_3}$ is overtwisted.
\end{proof}

Let $K$ be a null-homologous transverse knot in $(M, \xi_{k_1, k_2, k_3})$.
By Theorem~\ref{elena} we can identify $K$ with a closed $n$-braid $b$ in $(S, D_1^{k_1} \circ D_2^{k_2} \circ D_3^{k_3})$. 
\begin{assumption}\label{position-assumption-2}
Applying braid isotopy (transverse isotopy), we may assume that there exist points $x_1, \cdots, x_n$ (orange dots in Figure~\ref{3-punk-disk}) sitting between $\gamma_1$ and $\alpha_1$ such that
$b \cap (S\times \{0\}) = \{ x_1, \cdots, x_n \}.$
\end{assumption}

Let $\sigma_i$ be a mapping class of $S$ with $n$ fixed points $x_1, \cdots, x_n$, exchanging $x_i$ and $x_{i+1}$ counter clockwise.
Let $\rho_j$ ($j=1, 2, 3$) be a mapping class which moves $x_n$ around the boundary circle $\gamma_j$ as described in Figure~\ref{3-punk-disk}.
Since $\sigma_i$'s and $\rho_j$'s are generators of the mapping class group and $\rho_1=\rho_2+\rho_3,$ it follows that:

\begin{proposition}
An $n$-strand closed braid in the open book $(S, D_1^{k_1} \circ D_2^{k_2} \circ D_3^{k_3})$ can be written in letters $\{ \sigma_1, \cdots, \sigma_{n-1}, \rho_2, \rho_3 \}$.
\end{proposition}

We define symbols $a_\sigma, a_{\rho_i}$, and $s_i$:
\begin{definition}\label{def of a_rho_i}
Let $a_\sigma$ be the exponent sum of $\sigma_i$'s in the braid word for $b$.
Let $a_{\rho_i}$ ($i=2, 3$) be the exponent sum of $\rho_i$ in the braid word for $b$.
Since $b$ is null-homologous, we have $0= [b] = a_{\rho_2} c_2 + a_{\rho_3} c_3$ in $H_1(M;\Z)$. 
By equation (\ref{prop-homology}), there exist $s_2, s_3 \in \Z$, so that
$$0=[b] = s_2 \left\{ (k_1+k_2) c_2 + k_1 c_3 \right\} +
                 s_3 \left\{ k_1 c_2 + (k_1+k_3) c_3 \right\}.$$
Therefore,
\begin{equation}\label{eq matrix}
[a_{\rho_2}, a_{\rho_3}] = [s_2, s_3] \left[
\begin{array}{cc}
k_1+k_2 & k_1 \\
k_1 & k_1 + k_3
\end{array}
\right].
\end{equation}
For special cases;
\begin{enumerate}
\item
when $k_1=k_2=0$ and $k_3\neq 0$, we set $s_2=0$, i.e., $a_{\rho_2}=0$ and $a_{\rho_3}=s_3 k_3,$
\item
when $k_1=k_3=0$ and $k_2\neq 0$, we set $s_3=0$, i.e., $a_{\rho_3}=0$ and $a_{\rho_2}=s_2 k_2,$
\item
when $k_1=k_2=k_3=0$, we set $s_2=s_3=0$, i.e., $a_{\rho_2}=a_{\rho_3}=0$.
\end{enumerate}
\end{definition}

\begin{lemma}\label{s_i positive}
We may assume that $s_2, s_3 \geq 0$ and that $a_{\rho_2}, a_{\rho_3}$ satisfy {\em(\ref{eq matrix})}.
\end{lemma}

\begin{proof}

A similar argument as in the proof of Proposition~\ref{assumption for a_rho} applies.
Recall that a positive braid stabilization preserves the transverse knot type (Theorem~\ref{elena2}).
Since the point $x_n$ is between $\gamma_1$ and $\alpha_1$, positive stabilizations of $b$ about $\gamma_2$ (resp. $\gamma_3$) for $\alpha$-times (resp. $\beta$-times), where $\alpha, \beta \geq 0$, change $a_{\rho_i}$ in the following way:
$$a_{\rho_2} \mapsto a_{\rho_2} + \alpha(k_1 + k_2), \
a_{\rho_3} \mapsto a_{\rho_3} + \alpha k_1,
$$
$$
(\mbox{resp. } a_{\rho_2} \mapsto a_{\rho_2} + \beta k_1, \
a_{\rho_3} \mapsto a_{\rho_3} + \beta(k_1 + k_3)). 
$$
By (\ref{eq matrix}), $s_i$ also changes as;
$$s_2 \mapsto s_2 + \alpha \  \mbox{ if } k_1+k_2\neq 0 \mbox{ or } k_1\neq 0,$$
$$\mbox{(resp. } s_3 \mapsto s_3 + \beta \  \mbox{ if } k_1+k_3\neq 0 \mbox{ or } k_1\neq 0).$$
Therefore taking $\alpha, \beta$ sufficiently large, we can make $s_2, s_3 \geq 0$, even for the special three cases in Definition~\ref{def of a_rho_i}. 
\end{proof}

Now we state our main result of this section:

\begin{theorem}\label{sl for Seifert manifold}
Let $b$ be a null-homologous closed braid in $(S, D_1^{k_1} \circ D_2^{k_2} \circ D_3^{k_3})$ satisfying Assumption~\ref{position-assumption-2} and Lemma~\ref{s_i positive}. 
There is a Seifert surface $\Sigma_b$ of $b$ for which the following holds:
$$
sl(b, [\Sigma_b])=-n+a_{\sigma}+a_{\rho_2}(1-s_2) + a_{\rho_3}(1-s_3) - (s_2 + s_3) k_1.
$$
\end{theorem}

\begin{proof}

We construct an embedded surface $\tilde F_b$ after Sections~\ref{section F_b} and \ref{section-tilde F_b}.
\begin{itemize}

\item 
Construct $n$-copies of the $\delta$-disk (cf. Figure~\ref{disks}).

\item 
Join them by twisted bands for each $\sigma_j^{\pm}$ in the braid word (cf. Figure~\ref{surface1}-(1)).

\item 
Attach ${\mathfrak A}$-annuli for each $\rho_2^{\pm}, \rho_3^{\pm}$ in the braid word (cf. Figure~\ref{surface1}-(2)).

\item 
Attach vertical nested annuli to remove redundant boundaries (cf. Figure~\ref{canceling_rho}). 
This procedure is more subtle than that of annulus open book case.

If $k_1, k_2, k_3 \geq 0$ or $k_1, k_2, k_3 \leq 0$ then attach vertical nested annuli near the bindings $\gamma_2$ and $\gamma_3$ following the algorithm described in the proof of Proposition~\ref{a}.
The boundary of the resulting surface, $\tilde{F}_b,$ has 
\begin{eqnarray*}
\partial\tilde F_b &=&
b\ \cup \
\{ |(s_2+s_3) k_1 + s_2 k_2|  \mbox{ copies of } \epsilon(k_2) \beta_2 \} \\
&& \cup \
\{ |(s_2+s_3) k_1 + s_3 k_3|  \mbox{ copies of } \epsilon(k_3) \beta_3 \}
\end{eqnarray*}
where $\beta_i$ is the oriented circle as in Figure~\ref{3-punk-disk} and $\epsilon(k_i)$ is the sign of $k_i$.

Otherwise, by symmetry of the pants surface we may assume that 
(i) $k_1, k_2 \geq 0$ and $k_3<0$, or
(ii) $k_1, k_2 \leq 0$ and $k_3>0.$
For either case, we change the braid word by adding dummy letters that preserves the transverse knot type:
\begin{equation}\label{dummy}
b \mapsto (b\ \rho_3^{-s_3 k_3})\ (\rho_3^{s_3 k_3})
\end{equation}
Attach vertical nested annuli to the first part $b \rho_3^{-s_3 k_3}$, without touching the remaining part $\rho_3^{s_3 k_3}$,  until all their boundary circles of the $\mathfrak A$-annuli have the same direction. 
The boundary of the resulting surface, $\tilde{F}_b,$ has 
\begin{eqnarray*}
\partial\tilde F_b &=&
b\ \cup \
\{ (s_2+s_3)|k_1|   \mbox{ copies of } \epsilon(k_2) (\beta_2 + \beta_3) \} \\
&& \cup \ 
\{ s_2 |k_2| \mbox{ copies of } \epsilon(k_2)\beta_2 \} \
\cup \
\{ s_3 |k_3| \mbox{ copies of } \epsilon(k_3)\beta_3 \}.
\end{eqnarray*}
\end{itemize}

We require the operation (\ref{dummy}) so that $\mathcal D$-disks introduced below can be compatible with the monodromy of the open book. 
Note that if the signs of $k_2$ and $k_3$ are different
$$|a_{\rho_2}|=|(s_2+s_3) k_1 + s_2 k_2| = (s_2+s_3) |k_1| + s_2 |k_2|,$$
but
$$|a_{\rho_3}|=|(s_2+s_3) k_1 + s_3 k_3| \neq (s_2+s_3) |k_1| + s_3 |k_3|.$$
Compare Figure~\ref{immersed-surface2}, where $k_i$'s have the same sign, and Figure~\ref{immersed-surface4}, where $k_1, k_2>0$ and $k_3<0$. Their right bottom parts are different.

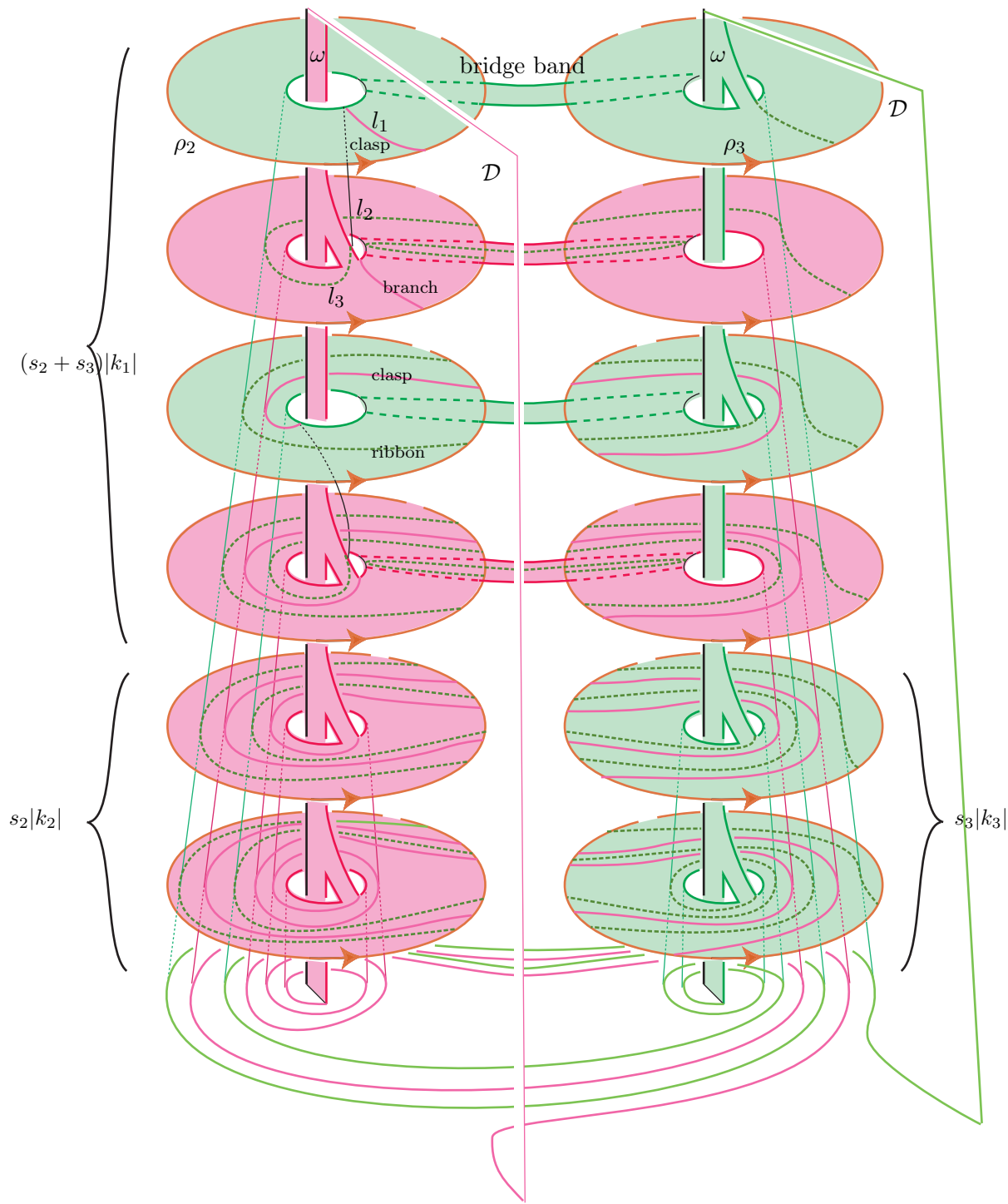
\begin{figure}[htpb!]
\begin{center}
\begin{picture}(400, 540)
\put(0,0){\includegraphics{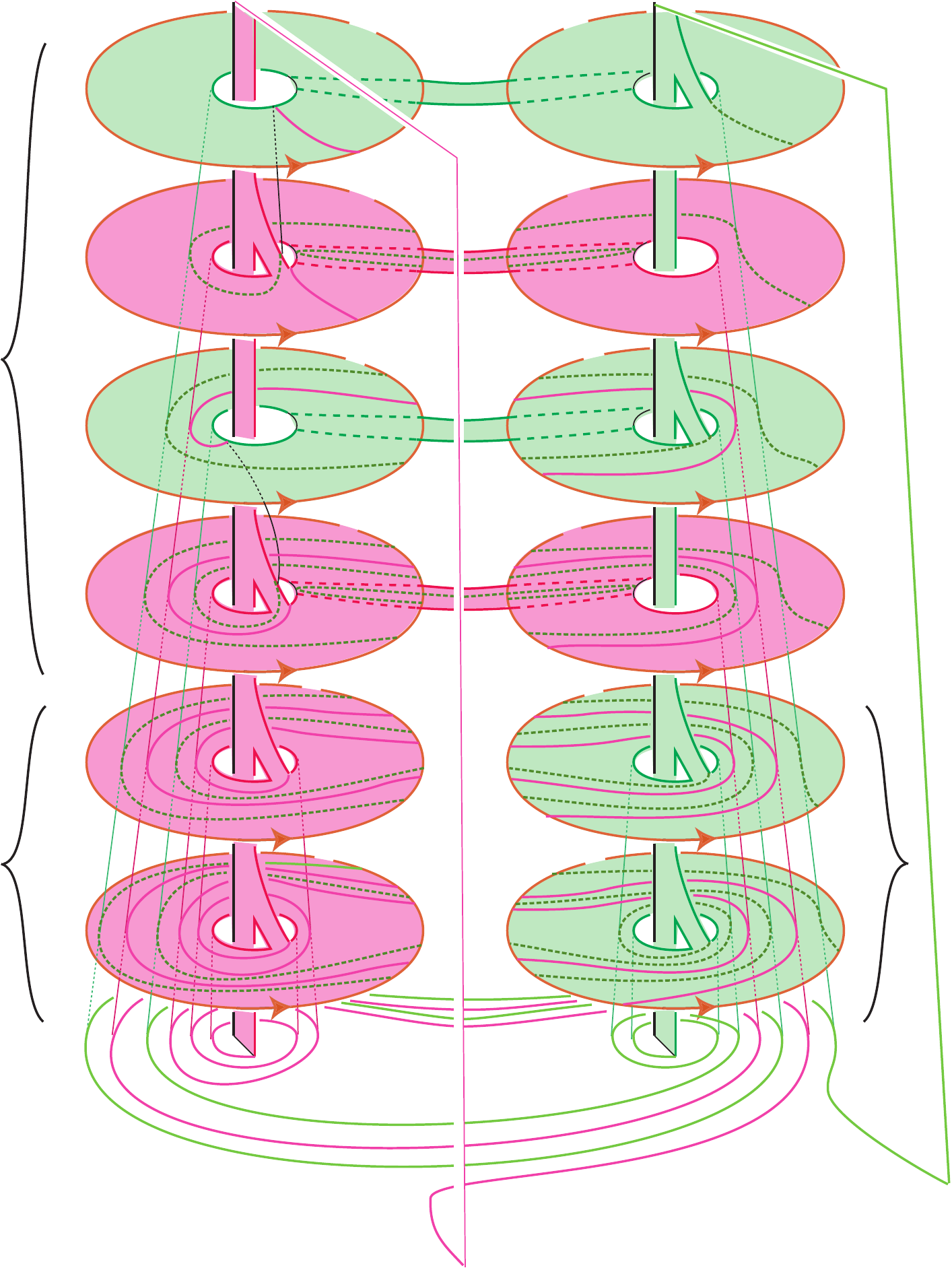}}
\put(102, 520){$\omega$}
\put(284, 520){$\omega$}
\put(180, 466){$\mathcal D$}
\put(365, 495){$\mathcal D$}
\put(-30, 380){\small $(s_2+s_3)|k_1|$}
\put(-35, 173){\small $s_2|k_2|$}
\put(395, 173){\small $s_3|k_3|$}
\put(130, 490){$l_1$}
\put(123, 450){$l_2$}
\put(110, 410){$l_3$}
\put(120, 480){\scriptsize clasp}
\put(135, 415){\scriptsize branch}
\put(130, 375){\scriptsize clasp}
\put(130, 340){\scriptsize ribbon}
\put(170, 515){bridge band}
\put(40, 480){$\rho_2$}
\put(290, 480){$\rho_3$}

\end{picture}
\caption{An immersed surface. $k_1=k_2=k_3=2$. $s_2=s_3=1$. Bridge bands are attached from above/below $\mathfrak A$-annuli depending on their $\theta$-heights.}\label{immersed-surface2}
\end{center}
\end{figure}

\begin{figure}[htpb!]
\begin{center}
\begin{picture}(400, 540)
\put(0,0){\includegraphics{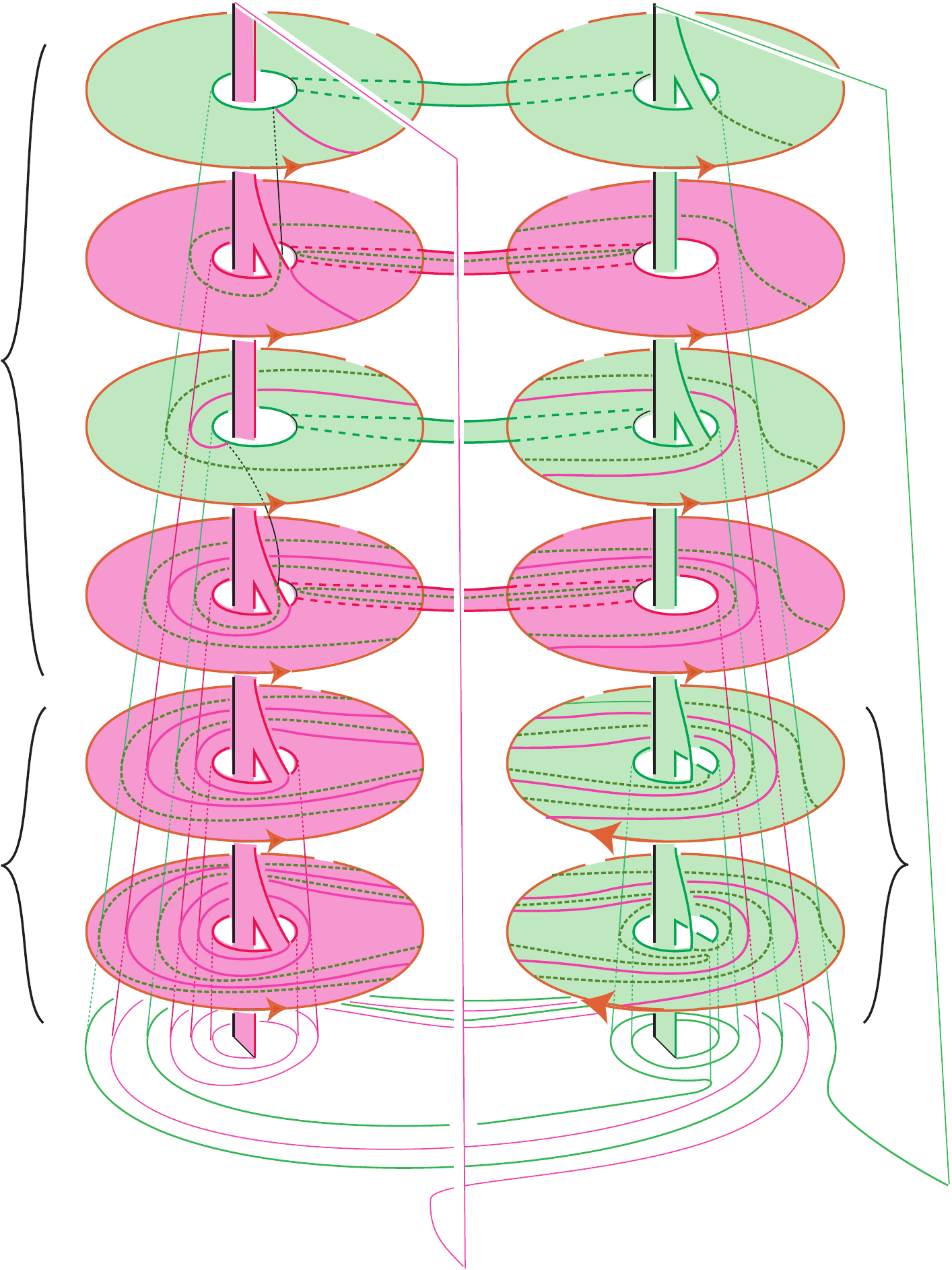}}

\put(102, 520){$\omega$}
\put(284, 520){$\omega$}
\put(180, 466){$\mathcal D$}
\put(365, 495){$\mathcal D$}
\put(-30, 375){\small $(s_2+s_3) |k_1|$}
\put(-35, 173){\small $s_2|k_2|$}
\put(395, 173){\small $s_3|k_3|$}
\put(40, 480){$\rho_2$}
\put(290, 480){$\rho_3$}

\end{picture}
\caption{An immersed surface. $k_1=k_2=2, k_3=-2$, $s_2=s_3=1$.}\label{immersed-surface4}
\end{center}
\end{figure}

Next step is to construct an immersed surface ${\hat F}_b$. 
If $k_1=k_2=k_3=0$ we define $\hat F_b = \tilde F_b$. Otherwise apply the following:
\begin{itemize}

\item 
If $k_1 =0$ we skip this step. 

If $k_1\neq 0$, join the top two $\mathfrak A$-annuli around $\gamma_2$ and $\gamma_3$ by a band, called a {\em bridge band} as in Figures~\ref{foliation-3}, \ref{immersed-surface2}.
The bridge band connects the $\beta_2$ circle and the $\beta_3$ circle of $\tilde F_b$.
If the $\theta$-coordinate (height) of the $\mathfrak A$-annulus around $\gamma_2$ is larger (resp. smaller) than the one around $\gamma_3$, then the bridge band is attached to $\beta_2$ from below (resp. above) and to $\beta_3$ from above (resp. below).
As a consequence, the bridge band does not tangent to the pages of the open book.

Repeat this for the first $(s_2+s_3)|k_1|$ pairs of $\mathfrak A$-annuli from the top.

\item 
If $k_1=k_2=0$ we skip this step. 

Otherwise, put $s_2$ copies of the $\omega$-disk, $\omega_1, \cdots, \omega_{s_2}$, about $\gamma_2$ (dark pink in Figure~\ref{immersed-surface2}).
Denote the $\mathfrak A$-annuli around $\gamma_2$ from the top by $\mathfrak A_1, \cdots, \mathfrak A_{|k_1|(s_2+s_3)+|k_2|s_2}$. 
Connect the disk $\omega_i$ with $|k_1|+ |k_2|$ copies of ${\mathfrak A}$-annulus; 
\begin{eqnarray*}
&& 
\mathfrak A_{s_3+i},\  
\mathfrak A_{2s_3+s_2+i},\ \cdots,\  
\mathfrak A_{|k_1|s_3+(|k_1|-1)s_2+i}, \\
&& 
\mathfrak A_{|k_1|(s_3+s_2)+i},\  
\mathfrak A_{|k_1|(s_3+s_2)+s_2+i},\ \cdots,\ 
\mathfrak A_{|k_1|(s_3+s_2)+(|k_2|-1)s_2+i},    
\end{eqnarray*}
\begin{eqnarray*}
&&
\mbox{(when } k_1=0, k_2 \neq 0, \mbox{ they are }
\mathfrak A_i,\ 
\mathfrak A_{s_2+i},\ \cdots,\ 
\mathfrak A_{(|k_2|-1)s_2 +i}, \\
&&
\mbox{when } k_1\neq 0, k_2 =0, \mbox{ they are }
\mathfrak A_{s_3+i},\  
\mathfrak A_{2s_3+s_2+i},\ \cdots,\  
\mathfrak A_{|k_1|s_3+(|k_1|-1)s_2+i})
\end{eqnarray*}
by using the twisted bands. Depending on the signs of $k_1$ and $k_2$, the twisted band is attached differently as described in Figure~\ref{foliation-1}.

\item 
Attach $s_2$ copies of $\mathcal{D}$-disk about $\gamma_1$ to the $\mathfrak A$-annuli around $\gamma_2$.

\item
Similarly, attach $s_3$ copies of $\omega$-disk (lighter shaded in Figure~\ref{immersed-surface2}) about $\gamma_3$, add $|k_1|+|k_3|$ copies of twisted bands, and $s_3$ copies of $\mathcal{D}$-disk.
\end{itemize}
Finally we have obtained an immersed surface ${\hat F}_b$ with boundary $b$.

The following two lemmas investigate singularities of the characteristic foliation.

\begin{lemma}\label{lemma h-band}

If $k_1 >0$ $($resp. $k_1<0)$ then the characteristic foliation for each bridge band has a single negative $($resp. positive$)$ hyperbolic singularity. See Figure~\ref{foliation-3}.

\end{lemma}

\begin{figure}[htpb!]
\begin{center}
\begin{picture}(275, 100)
\put(0,0){\includegraphics{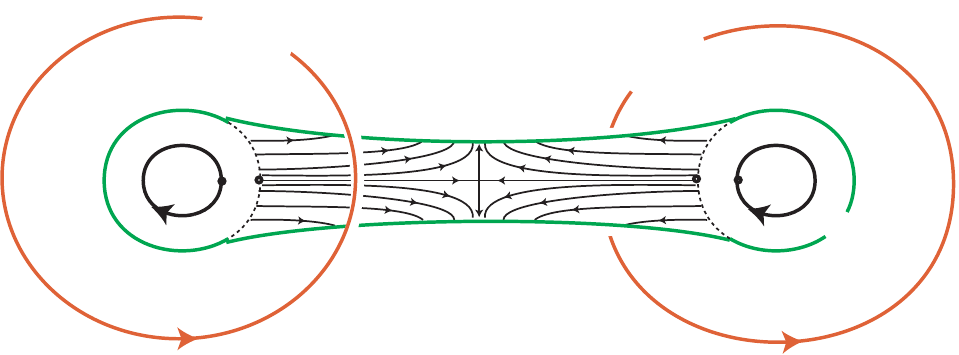}}
\put(30, 10){\scriptsize$\rho_2$}
\put(205, 10){\scriptsize$\rho_3$}
\put(45, 60){\scriptsize$\gamma_2$}
\put(220, 60){\scriptsize$\gamma_3$}
\put(52, 50){\scriptsize$p_2$}
\put(215, 50){\scriptsize$p_3$}
\put(80, 50){\scriptsize$p_2'$}
\put(190, 50){\scriptsize$p_3'$}

\put(10, 50){\scriptsize${\mathfrak A}$}
\put(260, 50){\scriptsize${\mathfrak A}$}
\put(25, 50){\scriptsize$\beta_2$}
\put(245, 50){\scriptsize$\beta_3$}
\end{picture}
\caption{Characteristic foliation on a bridge band when $k_1>0$.}\label{foliation-3}
\end{center}
\end{figure}

\begin{proof}

Let $p_i$ ($i=2,3$) be a point on the binding $\gamma_i$. Assume $p_i' \in \beta_i$ is a point close to $p_i$ and the bridge band connects $p_2'$ and $p_3'$, see Figure~\ref{foliation-3}.
At $p_i$, the contact plane intersects $\gamma_i$ positively.
Along the line segment from $p_2'$ to $p_3'$, the contact planes rotate $(180-\epsilon)^\circ$ counterclockwise. The contact plane is tangent to the bridge band at a single point somewhere between $p_2'$ and $p_3'$, where the hyperbolic singularity occurs.  If $k_1>0$ (resp. $k_1<0$), the positive (negative) side of the band is facing up to the reader, thus the sign of the hyperbolic point is positive (negative).
\end{proof}

\begin{lemma}
All the singularities of the characteristic foliation for $(\hat{F}_b \setminus \tilde{F}_b)$, the union of $\omega$-disks, twisted bands and $\mathcal D$-disks, are elliptic. Moreover, the algebraic count $e^+ -e^-$ of elliptic singularities for the $\omega$-disks $($resp. $\mathcal D$-disks$)$ is $s_2+s_3$, $($resp. $-s_2-s_3)$.
\end{lemma}

\begin{proof}
As seen in the proof of Lemma~\ref{no singularity on band}, there are no hyperbolic singularities on the twisted bands.
\end{proof}


We continue the proof of Theorem~\ref{sl for Seifert manifold}.

Figure~\ref{immersed-surface2} exhibits branch, clasp and ribbon intersections of $\hat{F}_b$.
For example, in Figure~\ref{immersed-surface2}, the union of arcs $l_1 \cup l_2 \cup l_3$ is a clasp intersection. 
Each pair among the $s_2+s_3$ $\mathcal D$-disks about $\gamma_1$ gives rise to $|k_1|$ clasp intersections.
Additionally, each pair among  the $s_2$ $\mathcal D$-disks about $\gamma_1$ attached to curves near $\gamma_2$ gives rise to $|k_2|$ clasp intersections, and  each pair among the $s_3$ $\mathcal D$-disks about $\gamma_1$ attached to curves near $\gamma_3$ gives rise to $|k_3|$ clasp intersections.
In total, there are 
\begin{eqnarray*}
&& |k_1| \binom{s_2+s_3}{2} + |k_2| \binom{s_2}{2} + |k_3| \binom{s_3}{2} \\
&=& \binom{s_2}{2}(|k_1| + |k_2|) +\binom{s_3}{2}(|k_1| + |k_3|) +s_2 s_3 |k_1|
\end{eqnarray*} 
clasp intersections.
Signs are assigned to each intersection according to Definition~\ref{def sign of singularity}.
If we count them algebraically,
\begin{eqnarray}
&&\mbox{\small algebraic number of branches} = -s_2(k_1 + k_2) - s_3(k_1+ k_3). \label{i1} \\
&&\mbox{\small  algebraic number of clasps} = -\binom{s_2}{2} (k_1 + k_2) - \binom{s_3}{3} (k_1 + k_3) - s_2 s_3 k_1. \label{i2}
\end{eqnarray}
Resolving all the intersection arcs, we obtain an embedded surface $\Sigma_b$.

By Proposition~\ref{prop-sign} and Corollary~\ref{cor-sign}, the
resolution of branch, clasp and ribbon intersections create additional hyperbolic singularities. The total algebraically counted number of such hyperbolic points is:
\begin{equation}\label{algebraic-count}
(\ref{i1}) + 2\times (\ref{i2}) =-\{(s_2 + s_3)^2 k_1 + s_2^2 k_2 +
s_3^2 k_3\}.
\end{equation}
In summary we have:
\begin{eqnarray*}
e^+ &=& (n; \small\mbox{ $\delta$-disks}) + (s_2 + s_3; \small\mbox{  $\omega$-disks}), \\
e^- &=& (s_2 + s_3; \small\mbox{ $\mathcal D$-disks}), \\
h^+ -h^- &=& (a_\sigma; \small\mbox{ bands between $\delta$-disks}) + (a_{\rho_2}+a_{\rho_3}; \small\mbox{ bands between $\delta_n$ and ${\mathfrak A}$-annuli}) \\
&& -(k_1(s_2+s_3); \small\mbox{ bridge bands; Lemma~\ref{lemma h-band}}) \\
&& - ((s_2 + s_3)^2 k_1 + s_2^2 k_2 + s_3^2 k_3; \small\mbox{ resolution of branches, clasps, ribbons (\ref{algebraic-count})}) \\
&\stackrel{(\ref{eq matrix})}{=}& a_\sigma + a_{\rho_2}+a_{\rho_3} - s_2 (a_{\rho_2} + k_1) - s_3 (a_{\rho_3} + k_1) \\
&=& a_\sigma + a_{\rho_2}(1-s_2) +a_{\rho_3}(1-s_3) - (s_2+s_3)k_1.
\end{eqnarray*}
Finally we have
\begin{eqnarray*}
sl(b, [\Sigma_b])&=&-(e^+-e^-)+(h^+-h^-) \\
&=& -n + a_\sigma + a_{\rho_2}(1-s_2) +a_{\rho_3}(1-s_3) - (s_2+s_3)k_1.
\end{eqnarray*}
\end{proof}

\bibliographystyle{amsplain}

\end{document}